\documentclass{article}

\input{epsf.tex}
\usepackage{amsfonts,amssymb,verbatim,amsmath,amsthm,latexsym,textcomp,amscd}
\usepackage{latexsym,amsfonts,amssymb,epsfig,verbatim}
\usepackage{amsmath,amsthm,amssymb,latexsym,graphics,textcomp}
\usepackage{graphicx}
\usepackage{color}
\usepackage{url}

\input xy
\xyoption{all}

\theoremstyle{plain}
\newtheorem{theorem}{Theorem}[section]

\newtheorem{lemma}[theorem]{Lemma}
\newtheorem{corollary}[theorem]{Corollary}
\newtheorem{remark}[theorem]{Remark}

\newtheorem{proposition}[theorem]{Proposition}

\newtheorem*{claim}{Claim}
\newtheorem*{subclaim}{Subclaim}

\def\Len{\mbox{\rm{\underline{length}}}}
\def\Mod{\mbox{\rm{Mod}}}
\def\ML{{\mathcal {ML}}}
\def\AF{{\mathcal{AF}}}
\def\PML{{\mathbb P}{\mathcal {ML}}}
\def\PAF{{\mathbb P}{\mathcal{AF}}}

\def\EL{{\mathcal {EL}}}

\def\T{{\mathcal T}}

\def\F{{\mathcal F}}

\def\C{{\mathcal C}}
\def\G{{\mathcal G}}
\def\B{{\mathcal B}}
\def\Diff{\mbox{\rm{Diff}}}
\def\id{\mbox{\rm{Id}}}
\def\A{{\mathcal A}}
\def\H{{\mathcal H}}
\def\ev{\mbox{\rm{ev}}}
\def\Conf{\mbox{\rm{Conf}}}

\def\E{{\mathcal E}}
\def\DD{{\mbox{\rm{DD}}}}
\def\PSL{{\mbox{\rm{PSL}}}}

\def\MF{{\mathcal {MF}}}
\def\PMF{{\mathbb P}{\mathcal {MF}}}

\def\best{{\text{best}}}
\def\arat{{\text{ar}}}
\def\zero{\ast}
\def\Zeros{{\mbox{\rm{Zeros}}}}

\def\re{{\mbox{\rm{Re}}}}
\def\AH{{\mbox{\rm{AH}}}}
\def\QF{{\mbox{\rm{QF}}}}
\def\Isom{{\mbox{\rm{Isom}}}}

\newcommand{\Per}{{\text{Per}}}
\newcommand{\I}{{\textbf{i}}}
\newcommand{\sing}{{\text{sing}}}
\newcommand{\ver}{{\text{Vert}}}

\title{Connectivity of the space of ending laminations.}
\author{Christopher J. Leininger\thanks{partially supported by NSF grant DMS-0603881} \& Saul Schleimer\thanks{partially
supported by NSF grant DMS-0508971}}

\begin{document}

\maketitle

\begin{abstract}We prove that for any closed surface of genus at least four, and any punctured surface of genus at least
two, the space of ending laminations is connected.  A theorem of E. Klarreich implies that this space is homeomorphic
to the Gromov boundary of the complex of curves.  It follows that the boundary of the complex of curves is connected in
these cases, answering the conjecture of P.~Storm.  Other applications include the rigidity of the complex of curves
and connectivity of spaces of degenerate Kleinian groups.
\end{abstract}

\section{Introduction}

Let $\Sigma = \Sigma_{g,n}$ be an orientable surface with genus $g$ and $n$ marked points. The space of measured
foliations on $\Sigma$ is denoted $\MF(\Sigma)$.  A measured foliation is \textbf{arational} if there are no leaf
cycles; see Section \ref{foliationsection} (such foliations are necessarily minimal).  We denote the space of arational
measured foliations on $\Sigma$ by $\AF(\Sigma) \subset \MF(\Sigma)$ and give it the subspace topology.  The space
$\AF(\Sigma)$ is related to the complex of curves and spaces of hyperbolic 3-manifolds as described below.

\begin{theorem} \label{mainarational}
If $g \geq 4$ or $g \geq 2$ and $n \geq 1$, then $\AF(\Sigma_{g,n})$ is connected.
\end{theorem}

We note that for $\Sigma = \Sigma_{1,1}$ and $\Sigma_{0,4}$, the space $\AF(\Sigma)$ is naturally homeomorphic to
$\mathbb R - \mathbb Q$, and so is totally disconnected.

To describe some of the applications of this theorem, we recall that the work of Masur and Minsky \cite{MM1} together
with that of Klarreich \cite{klarreich} (see also \cite{bowditch,hamenstadtboundary}) implies that Harvey's complex of
curves $\C(\Sigma)$ is $\delta$--hyperbolic, and the Gromov boundary is homeomorphic to the quotient
\[ \partial \C(\Sigma) \cong \AF(\Sigma)/\hspace{-.1cm}\sim . \]
Here $\sim$ denotes the equivalence relation obtained by forgetting the transverse measures.  We thus have the
following corollary of Theorem \ref{mainarational}, affirmatively answering the Storm Conjecture for most surfaces (see
\cite[Question 10]{kentleininger} and \cite[page 8]{minskyICM}).

\begin{corollary} \label{C:storm}
If $g \geq 4$ or $g \geq 2$ and $n \geq 1$, then $\partial \C(\Sigma_{g,n})$ is
connected.
\end{corollary}

Connectivity of $\partial \C(\Sigma)$ is a useful property when trying to understand the quasi-isometric geometry of
$\C(\Sigma)$, as was shown by the second author and Rafi \cite{rafischleimer}.  Recall the measure of complexity
$\xi(\Sigma_{g,n}) = 3g - 3 + n$.

\begin{theorem} [Rafi--Schleimer]
Suppose $\partial \C(\Sigma)$ is connected and $\xi(\Sigma) \geq 4$.  If $\Sigma'$ is any surface for which
$\C(\Sigma)$ and $\C(\Sigma')$ are quasi-isometric, then $\Sigma \cong \Sigma'$.
\end{theorem}
A version of this theorem holds for $\xi(\Sigma) \geq 2$: a quasi-isometry from $\C(\Sigma)$ to $\C(\Sigma')$ is a
bounded distance from a simplicial isomorphism.  For $\xi(\Sigma) =2,3$ there are nonhomeomorphic surfaces with
isomorphic curve complexes.

Now view $\Sigma = \Sigma_{g,n}$ as a surface of genus $g$ with $n$ boundary components.  The natural homeomorphism
between $\MF(\Sigma)$ and $\ML(\Sigma)$, the space measured laminations on $\Sigma$, sends $\AF(\Sigma)$ to the
subspace of those laminations which are \textbf{filling}.  The quotient of this space by forgetting the transverse
measures is called the space of \textbf{ending laminations}, and is denoted $\EL(\Sigma)$.  Via the natural
homeomorphism we see that the following are all homeomorphic
\[\EL(\Sigma) \cong \AF(\Sigma)/\hspace{-.1cm}\sim \, \cong \partial \C(\Sigma).\] The space $\EL(\Sigma)$ is precisely the set of ending
laminations of geometrically infinite Kleinian surface groups isomorphic to $\pi_1(\Sigma)$ without accidental
parabolics.

The following is an easy consequence of Theorem \ref{mainarational} and its proof (see Theorems \ref{nocut1} and \ref{nocut2}).

\begin{theorem} \label{endinglams}
The space $\EL(\Sigma_{g,n})$ is connected and has no cut points when $g \geq 4$ or $g \geq 2$ and $n \geq 1$.
\end{theorem}

The Ending Lamination Theorem of Brock--Canary--Minsky \cite{ELCI},\cite{ELCII} asserts that the end invariants $\E(M)
= (\E^-(M),\E^+(M))$ are complete invariants of the hyperbolic manifold $M \in \AH(\Sigma,\partial \Sigma)$.  Moreover,
on the space of doubly degenerate Kleinian surface groups
\[ \DD(\Sigma,\partial \Sigma) = \{ M \in \AH(\Sigma,\partial \Sigma) \, | \, \E(M) \in \EL(\Sigma) \times \EL(\Sigma)\}, \]
the map $\E$ is a homeomorphism (Theorem \ref{ELThomeo1}).  Although this seems to be well-known, we provide a proof in
Section \ref{hyperbolicgeometry} for completeness. Thus one has the following.

\begin{proposition}
The space $\DD(\Sigma_{g,n})$ is connected if $g \geq 4$ or $g \geq 2$ and $n \geq 1$.
\end{proposition}

\begin{remark}
On all of $\AH(\Sigma,\partial \Sigma)$, the map $\E$ fails to be continuous with respect to any of the usual
topologies on the target (see \cite{brocklength}).  However, discontinuity occurs in the presence of accidental
parabolics which is not an issue for the situation being discussed here.
\end{remark}

A similar result holds if one considers the subset of the boundary of a Bers slice $\partial_0 B_Y \subset
\partial B_Y$ consisting of those Kleinian groups without accidental parabolics.  Then $\partial_0 B_Y \cong
\EL(\Sigma)$ (see Theorem \ref{ELThomeo2}). So, as a corollary of Theorem \ref{endinglams} or \ref{mainarational}, we
have

\begin{corollary}
The space $\partial_0 B_Y$ is connected if $g \geq 4$ or $g \geq 2$ and $n \geq 1$.
\end{corollary}

Finally, we remark that Theorem \ref{mainarational} has some bearing on problems regarding negatively curved
$4$--manifolds.  More precisely, an open question is whether or not there exists a closed, negatively curved
$4$--manifold $M$ which fibers over a surface (see Question 12.3(b) of \cite{bestvinaproblems}).  The fiber is also a
surface $\Sigma$, and a consequence of the work of Farb--Mosher \cite{FMcc} is that if such a manifold exists, then
there is an embedding of a circle to $\AF(\Sigma)$. One approach to prove that there are no such $4$-manifolds would
thus be to prove that $\AF(\Sigma)$ is totally disconnected, or at least, prove that $\AF(\Sigma)$ contains no circles.
Theorem \ref{mainarational} thus removes this obstruction---indeed, during the course of the proof of Theorem
\ref{mainarational}, we will construct many circles embedded in $\AF(\Sigma)$ (see Section \ref{cutpointsection1}).

\subsection{Plan of the paper}

The proof of Theorem \ref{mainarational} for surfaces with a single marked point is considerably simpler, and follows
from a suggestion made to us by Ken Bromberg.  Proceed as follows:  Fix an arational measured foliation on
$\Sigma_{g,0}$. Introducing a marked point yields an arational measured foliation on $\Sigma_{g,1}$ and moving the
point produces paths of such foliations. Indeed, the set of foliations so obtained forms a dense, path-connected subset
of $\AF(\Sigma_{g,1})$. Thus $\AF(\Sigma_{g,1})$ is connected.  The fact that the closure of a connected set is
connected will be exploited several times.

\begin{remark}
In joint work with Mahan Mj \cite{lms} we develop further tools, combining ideas from \cite{kls} and \cite{MahanG} to
provide a more precise picture for surfaces with one marked point.  As a consequence we prove that $\EL(\Sigma_{g,1})$
is path connected and locally path-connected.  We also note that using different techniques, David Gabai has now proven this for all surfaces
$\Sigma_{g,n}$ with $3g+n \geq 5$ \cite{gabai}.
\end{remark}

For surfaces $\Sigma_{g,n}$ with $n \geq 2$ a similar strategy can be employed.  However, the position of the marked
points is much more delicate.  Specifically, starting with an arational measured foliation on $\Sigma_{g,0}$, arbitrary
placement of the $n$ marked points does not result in an arational measured foliation on $\Sigma_{g,n}$. Consequently,
we must first devise an effective criteria which guarantees that the position of $n$ marked points determines an
arational measured foliation on $\Sigma_{g,n}$.  With such a criteria at our disposal, we can begin producing paths in
$\AF(\Sigma_{g,n})$.

Second, we must come to terms with the fact that the reduced flexibility in the placement of the marked points means
that we have fewer paths to work with.  In particular there is no obvious dense path-connected set.  We use the
dynamics of pseudo-Anosov mapping classes to produce paths which connect to pseudo-Anosov fixed points.  Concatenating
such paths we are able to produce a dense path-connected set, thus proving that $\AF(\Sigma_{g,n})$ is connected.

Our criteria for the positions of the marked points requires that we chose orientable foliations in $\AF(\Sigma_{g,0})$
when constructing paths.   Moreover, to construct our dense path-connected subset of $\AF(\Sigma_{g,n})$, we need some
connected space of (not necessarily arational) orientable foliations to get started.  The space of complex valued
$1$--forms (holomorphic with respect to varying complex structures on $\Sigma_{g,0}$) provides such a space.

To prove that $\AF(\Sigma_{g,0})$ is connected we apply a branched cover construction to find a family of connected
subsets so that (a) the union of the connected sets is dense and (b) for any two subsets in the family, there is a
finite chain of such subsets so that consecutive subsets in the chain nontrivially intersect.  It follows that the
union is a dense connected set, and hence $\AF(\Sigma_{g,0})$ is connected.

We end the paper by explaining the applications to hyperbolic 3-manifolds mentioned above in more detail.\\

\noindent {\bf Acknowledgements.} We wish to thank Ken Bromberg for suggesting the method of moving marked points, Jeff
Brock for convincing us to not give up on the branched cover approach for the closed case and both for many stimulating
conversations. We would also like to thank Gilbert Levitt for some enjoyable conversations regarding Theorem
\ref{mainarational}.  Finally, we would like to thank Cyril Lecuire and Jeff Brock for explaining to us the precise relation
with hyperbolic 3-manifolds as discussed in Section \ref{hyperbolicgeometry}.

\section{Preliminaries}

In this section we briefly describe the background material we will need, make some notational conventions (most of
which are standard), and make some preliminary observations.

We let $S$ denote a closed surface of genus at least $2$ and ${\bf z} = \{z_1,...,z_n\} \subset S$ a set with $n \geq
0$ points.  Because we will wish to refer to the marked points by name, we will write $(S,{\bf z})$ instead of
$\Sigma_{g,n}$.  We view ${\bf z} \subset S$ as an ordered set of distinct points.  We will sometimes refer to it as a
point in the $n$--fold product $S \times ... \times S$.

We will frequently make definitions for $(S,{\bf z})$ and consider them valid for $S$ unless they clearly only apply
when $|{\bf z}| \neq 0$.

\subsection{Curves and paths}

We will let $\C^0(S,{\bf z})$ denote the set of isotopy classes of essential simple closed curves contained in $S -
{\bf z}$. It is also convenient to define $\A^0(S,{\bf z}) \supset \C^0(S,{\bf z})$ by adding to $\C^0(S,{\bf z})$ the
set of all isotopy classes of essential arcs meeting ${\bf z}$ precisely in the endpoints of the arcs. Isotopies must
fix ${\bf z}$.  A curve or arc is essential if it cannot be isotoped into an arbitrarily small neighborhood of a point
of ${\bf z}$. The geometric intersection number $\I(\cdot,\cdot)$ is defined for pairs of points in $\A^0(S,{\bf z})$
as the minimal number of points of intersection between representatives of the curves/arcs.

We let $\Gamma(S)$ denote the set of homotopy classes of oriented closed curves on $S$, and $\Gamma(S,z_i,z_j)$ the set
of homotopy classes of oriented paths from $z_i$ to $z_j$.  Note that $\Gamma(S,z_i,z_j)$ and $\Gamma(S,z_j,z_i)$
differ simply by reversing the orientations on all homotopy classes.  This latter homotopy is relative to the
endpoints, but for example, the path may be homotoped through other marked points.

\subsection{Diffeomorphisms and mapping classes}

The orientation preserving diffeomorphism group of $S$ is denoted $\Diff^+(S)$.  There are several subgroups which we
are interested in:  $\Diff^+(S,{\bf z})$, the subgroup consisting of those diffeomorphisms fixing each $z_i \in {\bf
z}$, $\Diff_0(S)$ and $\Diff_0(S,{\bf z})$, the respective components containing the identity, as well as the intersection
$$\Diff_{0,{\bf z}}(S,{\bf z}) = \Diff_0(S) \cap \Diff^+(S,{\bf z}).$$

The \textbf{mapping class groups} we are interested in are
$$\Mod(S) = \Diff^+(S) / \Diff_0(S)$$
$$\Mod(S,{\bf z}) = \Diff^+(S,{\bf z}) / \Diff_0(S,{\bf z})$$
Given a diffeomorphism $\phi \in \Diff^+(S)$, we will denote its image in $\Mod(S)$ by $\bar \phi$. If $\phi \in
\Diff^+(S,{\bf z})$, then we denote its image in $\Mod(S,{\bf z})$ by $\hat \phi$.

We also have need to consider the $(S,{\bf z})$--\textbf{braid group}
$$\B(S,{\bf z}) = \Diff_{0,{\bf z}}(S,{\bf z})/ \Diff_0(S,{\bf z}) < \Mod(S,{\bf z}).$$
See below for the discussion of the Birman Exact Sequence and the connection to the usual definition of the surface braid group.

The mapping class groups act on the sets of curves and paths.  More precisely, $\Mod(S,{\bf z})$ acts on $\C^0(S,{\bf
z})$, $\A^0(S,{\bf z})$, $\Gamma(S)$, $\Gamma(S,z_i,z_j)$ in the usual way by pushing forward homotopy/isotopy classes.
We denote the result of the mapping class $\hat \phi$ acting on the homotopy/isotopy class of curve/path $\alpha$ by
$\hat \phi(\alpha)$.

\subsection{Configuration spaces}

The configuration space of $n$--ordered points on $S$ ($n \geq 1$) is the subspace of the $n$--fold product $S \times
... \times S$ obtained by removing the locus where two coordinates are equal:
$$\Conf_n(S) = \{(p_1,...,p_n) \, | \, p_i \in S \mbox{ for all } i \mbox{ and } p_i \neq p_j \mbox{ for all } i \neq j \}.$$
Observe that $\Conf_1(S) \cong S$ and for $n \geq 2$, $\Conf_n(S)$ fibers over $\Conf_{n-1}(S)$ with fibers
homeomorphic to $S$ with $(n-1)$ points removed.  Applying the long exact sequence of a fibration inductively, we see
that all higher homotopy groups of $\Conf_n(S)$ vanish.  It follows that the universal covering $\widetilde \Conf_n(S)$
is contractible.

We think of ${\bf z}$ as a basepoint for $\Conf_n(S)$. This determines an evaluation map
$$\ev_{\bf z}: \Diff_0(S) \to \Conf_n(S)$$
given by $\ev_{\bf z}(\phi) = \phi({\bf z})$. As in Birman's work \cite{birmansequence,birmanbook}, the group
$\Diff_{0,{\bf z}}(S,{\bf z})$ acts on the fibers, and makes $\Diff_0(S)$ into a principal $\Diff_{0,{\bf z}}(S,{\bf
z})$--bundle. We will use local trivializations for this fibration which we discuss in more detail in the next section.

The long exact sequence of homotopy groups of a fibration, together with the contractibility of $\Diff_0(S)$---due to
Earle and Eells \cite{earleeells}---gives isomorphisms
\begin{equation} \label{E:birmanisomorphism} \B(S,{\bf z}) = \pi_0(\Diff_{0,{\bf z}}(S,{\bf z})) \cong \pi_1(\Conf_n(S)). \end{equation}
This justifies our referring to $\B(S,{\bf z})$ as the braid group since the last group $\pi_1(\Conf_n(S))$ is the
usual definition for the (pure) $n$--strand braid group on $S$. This isomorphism and the short exact sequence below
were obtained by Birman \cite{birmansequence,birmanbook} for the homeomorphism group.

It follows that the quotient of $\Diff_0(S)$ by the smaller group $\Diff_0(S,{\bf z})$ is the universal cover
$\widetilde \Conf_n(S)$.  We thus obtain $\Diff_0(S)$ as a principal $\Diff_0(S,{\bf z})$--bundle over $\widetilde
\Conf_n(S)$.  Contractibility of $\widetilde \Conf_n(S)$ implies
\begin{equation} \label{E:product1}
\Diff_0(S) \cong \widetilde \Conf_n(S) \times \Diff_0(S,{\bf z}).
\end{equation}
The basepoint ${\bf z}$ for $\Conf_n(S)$ has a canonical lift $\tilde {\bf z}$ to $\widetilde \Conf_n(S)$, namely the image
of the identity in $\Diff_0(S)$.

The inclusion $\Diff^+(S,{\bf z}) < \Diff^+(S)$ induces a homomorphism
$$\Mod(S,{\bf z}) \to \Mod(S).$$
The discussion above, together with the isomorphism theorems from group theory situates this homomorphism into the
\textbf{Birman Exact Sequence}
$$1 \to \B(S,{\bf z}) \to \Mod(S,{\bf z}) \to \Mod(S) \to 1.$$
We use this to view $\B(S,{\bf z})$ as a subgroup of $\Mod(S,{\bf z})$.

\subsection{Local trivializations}  \label{S:trivialization}

We now describe the local trivializations (that is, local sections) for the principal bundles $\Diff_0(S) \to
\widetilde \Conf_n(S)$ and $\Diff_0(S) \to \Conf_n(S)$ which we will use.  We describe these only near the point ${\bf
z}$ as this is our primary case of interest.

Let $B_1,...,B_n$ be open disk neighborhoods of $z_1,...,z_n$, respectively, in $S$ with pairwise disjoint disk
closures $\overline B_1,...,\overline B_n$.  We let $U_1,...,U_n$ be pairwise disjoint open disks with $\overline B_i
\subset U_i$. Write ${\bf B} = B_1 \times ...\times B_n$ and ${\bf U} = U_1 \times ... \times U_n$ with points denoted
${\bf b} = (b_1,...,b_n)$.

Consider a smooth map
$$f:S \times {\bf B} \to S.$$
For ${\bf b} \in {\bf B}$, let $f_{\bf b} = f(\cdot,{\bf b}):S \to S$.  We suppose $f$ has the following properties.
\begin{itemize}
\item $f_{\bf z} = \id$,
\item $f_{\bf b}$ is a diffeomorphism for every ${\bf b} \in {\bf B}$,
\item $f_{\bf b}$ is the identity outside $\displaystyle{\cup_i U_i}$,
\item $f(z_i,{\bf b}) = b_i$ for every ${\bf b} \in {\bf B}$ and $i = 1,...,n$.
\end{itemize}

Note that ${\bf B} \subset \Conf_n(S)$ is a neighborhood of ${\bf z}$ and the map
$$f_{\bf B}:{\bf B} \to \Diff_0(S)$$
given by $f_{\bf B}({\bf b}) = f_{\bf b}$ is a local trivialization for $\Diff_0(S) \to \Conf_n(S)$ over this
neighborhood.  Similarly, this determines a local trivialization for $\Diff_0(S) \to \widetilde \Conf_n(S)$ over the
neighborhood of $\tilde {\bf z}$ obtained by lifting ${\bf B}$ to $\widetilde \Conf_n(S)$. We call either $f$ or
$f_{\bf B}$ a ${\bf B}$--trivialization.

Using local coordinates one can construct a ${\bf B}$--trivialization for any ${\bf U}$ and ${\bf B}$ as above.

\subsection{Measured foliations} \label{foliationsection}

We refer the reader to \cite{FLP} and \cite{moshertrain} for a more detailed discussion of measured foliations on
surfaces.  We remark that the definitions of measured foliations for surfaces with marked points (or punctures) is
treated in \cite{FLP} by replacing the puncture with a boundary component, and making all definitions on compact
surfaces with boundary.

A \textbf{measured foliation} on $(S,{\bf z})$ is a singular foliation $\F$ on $S$ together with a transverse measure
$\mu$ of full support.  The singularities of $\F$ are all required to be $p$--prong singularities for $p \geq 1$, or $p
\geq 3$ if the singularity does not occur at a marked point.  We denote the set of singularities by $\sing(\F) \subset
S$.

Given a measured foliation $(\F,\mu)$, and $\alpha \in \C^0(S,{\bf z})$, the geometric intersection number
$\I(\alpha,(\F,\mu))$ is defined as the infimum
$$\I(\alpha,(\F,\mu)) = \inf_{\alpha_0 \in \alpha} \int_{\alpha_0} \mu.$$
This is the infimum of the total variation of $\alpha_0$ as $\alpha_0$ ranges over all representatives of $\alpha$.

Two measured foliations $(\F,\mu)$ and $(\F',\mu')$ are declared to be equivalent on $(S,{\bf z})$ if
$$\I(\alpha,(\F,\mu)) = \I(\alpha,(\F',\mu'))$$
for all $\alpha \in \C^0(S,{\bf z})$.  Th\'eor\`eme 1 of Expos\'e 11 in \cite{FLP} states (in particular) that this is
the same as the equivalence relation on measured foliations generated by Whitehead equivalence and isotopy on $(S,{\bf
z})$.  We denote the space of equivalence classes by $\MF(S,{\bf z})$, topologized as a subspace of $\mathbb
R^{\C^0(S,{\bf z})}$ via the inclusion
$$[\F,\mu] \mapsto \{ \I(\alpha,(\F,\mu)) \}_{\alpha \in \C^0(S,{\bf z})}.$$

It will be convenient at times to denote the equivalence class of $(\F,\mu)$ by $\mu$ rather than $[\F,\mu]$.  Thus,
when we write $\mu$ we are referring to an equivalence class of measured foliation (even if we inappropriately call it
a measured foliation), where as the notation $(\F,\mu)$ means the actual measured foliation, not just the equivalence
class it determines.  We also write $\I(\alpha,\mu)$ to denote $\I(\alpha,(\F,\mu))$ when convenient.

A measured foliation $\mu \in \MF(S,{\bf z})$ is \textbf{orientable} if it has a representative which is transversely
orientable.  We remark that any transversely orientable foliation is Whitehead equivalent to one which is not
transversely orientable.  The reason is that a Whitehead move can turn an even prong singularity into a pair of odd
prong singularities.

We define a \textbf{saddle connection} of a foliation $\F$ to be the image of a path $\epsilon:I \to S$, defined on a
compact interval $I$, with the following properties:
\begin{itemize}
\item $\epsilon$ is injective and tangent to $\F$ on the interior of $I$,
\item $\epsilon$ maps the interior disjoint from $\sing(\F) \cup {\bf z}$, and
\item $\epsilon$ maps the endpoints of $I$ into $\sing(\F) \cup {\bf z}$.
\end{itemize}
A \textbf{leaf cycle} is an embedded loop or an embedded path connecting points of ${\bf z}$ which is a concatenation
of saddle connections.  A leaf cycle which is a loop is called a \textbf{closed leaf cycle}.

A measured foliation $(\F,\mu)$ is \textbf{arational} if $\F$ has no leaf cycle.  The existence of a leaf cycle is not
changed by Whitehead moves, and so we may say $\mu$ is arational if any representative is.  An equivalent formulation
is that a measured foliation $(\F,\mu)$ is arational if
$$\I(\alpha,(\F,\mu)) > 0 \mbox{ for every } \alpha \in \C^0(S,{\bf z}).$$

If a measured foliation $(\F,\mu)$ on $(S,{\bf z})$ has no $1$--prong singularities, then it determines points in both
$\MF(S,{\bf z})$ and $\MF(S) = \MF(S,\emptyset)$.  We write $\mu$ and $\pi(\mu)$ for these respective points.

\begin{lemma} \label{arationalprecriterion}
If $(\F,\mu)$ has no $1$--prong singularities and $\pi(\mu)$ is arational, then $\mu$ is arational if and only if it
has no leaf cycle connecting two distinct points $z_i,z_j \in {\bf z}$.
\end{lemma}
\begin{proof}
To distinguish whether we are viewing the foliation $\F$ on $S$ or $(S,{\bf z})$ we write $(\F,\pi(\mu))$ and
$(\F,\mu)$, respectively.  The only minor subtlety involved in the proof is that a saddle connection for $(\F,\mu)$
which has at least one endpoint in ${\bf z}$ is not necessarily a saddle connection for $(\F,\pi(\mu))$.

If $\mu$ is arational, then there are no leaf cycles, by definition.  In particular, there are no leaf cycles
connecting $z_i$ to $z_j$ for any $i \neq j$.

We prove the other implication by proving the contrapositive.  Suppose that $\mu$ is not arational so that there exists a leaf cycle $\gamma$.  Because $\gamma$ is embedded, one can check that one of the following must happen: (1) $\gamma$ is a closed leaf cycle for $(\F,\pi(\mu))$, (2) $\gamma$ is a closed leaf of $(\F,\pi(\mu))$, or (3) $\gamma$ is a non-closed leaf cycle.  Case (1) cannot happen since we are assuming that $(\F,\pi(\mu))$ is arational.  Likewise, case (2) implies that there is a cylinder, the boundary of which contains a closed leaf cycle, again ruled out by arationality of $(\F,\pi(\mu))$.  It follows that case (3) must occur, which is the desired conclusion.
\end{proof}

The group $\Mod(S,{\bf z})$ acts on $\MF(S,{\bf z})$ and this can be most easily described via the change in
intersection numbers. Specifically, if $\alpha$ is the isotopy class of a closed curve or arc, $\mu \in \MF(S,{\bf
z})$, and $\hat \phi \in \Mod(S,{\bf z})$ then
$$\I(\alpha, \hat \phi \cdot \mu) = \I(\hat \phi^{-1} (\alpha), \mu).$$
That is, the action should preserve geometric intersection number.

It is also convenient to have the space $\PMF(S,{\bf z})$ of \textbf{projective measured foliations}.  This is the
quotient of $\MF(S,{\bf z})$ by the action of $\mathbb R_+$ by scaling the transverse measure.  The action of
$\Mod(S,{\bf z})$ on $\MF(S,{\bf z})$ descends to an action on $\PMF(S,{\bf z})$.  An element of $\PMF(S,{\bf z})$ is
arational if and only if any (equivalently, every) of its preimages is.

If $\hat \phi \in \Mod(S,{\bf z})$ is pseudo-Anosov we let $\mathbb P(\mu_s),\mathbb P(\mu_u) \in \PMF(S,{\bf z})$
denote the \textbf{stable and unstable projective measured foliations} of $\phi$.   These are attracting and repelling
fixed points, respectively.  That is, on $\PMF(S,{\bf z}) - \{ \mathbb P(\mu_u) \}$, iteration of $\hat \phi$ converges
uniformly on compact sets to the constant map with value $\mathbb P(\mu_s)$.   Inverting $\hat \phi$ we obtain the same
dynamics after interchanging $\mathbb P(\mu_s)$ and $\mathbb P(\mu_u)$.  We also call $\mu_s$ and $\mu_u$ the stable
and unstable measured foliations of $\hat \phi$, though they are only well defined up the action of $\mathbb R_+$.

\begin{lemma} \label{nonorientstable} If $\hat \phi \in \B(S,{\bf z})$ is pseudo-Anosov and $\mu_s$ is its stable foliation, then any representative of $\mu_s$ has a $1$--prong singularity.
\end{lemma}
\begin{proof}
Let $\phi_0:(S,{\bf z}) \to (S,{\bf z})$ be a representative \textit{pseudo-Anosov homeomorphism} of $\phi$ with stable
and unstable measured foliations $(\F_s,\mu_s)$ and $(\F_u,\mu_u)$; see \cite{FLP}.  Suppose $\F_s$ has no $1$--prong
singularity. By transversality $\F_u$ does not have one either.  It follows that we can forget ${\bf z}$ and $\phi_0:S
\to S$ is still a pseudo-Anosov homeomorphism.  Therefore, the class in $\Diff^+(S)$ determined by $\phi_0$ is
pseudo-Anosov. However, $\hat \phi \in \B(S,{\bf z})$ means that any representative lies in $\Diff_0(S)$ and so is
trivial in $\Mod(S)$, and cannot be pseudo-Anosov on $S$.

Note that if $(\F_s',\mu_s')$ is any other representative of the class $\mu_s$, then $\F_s$ is obtained from $\F_s'$ by
isotopy and \textit{collapsing} Whitehead moves only.  This is because $\F_s$ can have no saddle connections.  If
$\F_s'$ had a $1$--prong singularity, then $\F_s$ must have also had a $1$--prong singularity.
\end{proof}

\subsection{Teichm\"uller space and holomorphic $1$--forms} \label{abeliansection}

To discuss the space of holomorphic $1$--forms, which is the primary space of interest for us, we first recall some
facts about Teichm\"uller space.  The space of complex structures on $S$, compatible with the smooth structure and
orientation, is denoted $\H(S)$. The group $\Diff^+(S)$ acts on $\H(S)$ \textit{on the right} by pulling back complex
structures, and the Teichm\"uller space of $S$ is the quotient by the action of the subgroup $\Diff_0(S)$
$$\T(S) = \H(S)/\Diff_0(S).$$
The action on the fibers of the map $\H(S) \to \T(S)$ is simply transitive giving $\H(S)$ the structure of a principal
$\Diff_0(S)$--bundle \cite{earleeells}.  Contractibility of $\T(S)$ implies
$$\H(S) \cong \T(S) \times \Diff_0(S).$$

Keeping track of the marked points ${\bf z} \subset S$ amounts to taking the quotient by the smaller group
$\Diff_0(S,{\bf z})$.  That is, the Teichm\"uller space of $(S,{\bf z})$ is
$$\T(S,{\bf z}) = \H(S) / \Diff_0(S,{\bf z}).$$
Combining this discussion with (\ref{E:product1}) we obtain
\begin{equation} \label{producteqn} \H(S) \cong \T(S)
\times \widetilde \Conf_n(S) \times \Diff_0(S,{\bf z}) \cong \T(S,{\bf z}) \times \Diff_0(S,{\bf z})
\end{equation}
and so
$$\T(S,{\bf z}) \cong \T(S) \times \widetilde \Conf_n(S).$$


\begin{remark}
In the case that ${\bf z}$ is a single marked point, Bers proved that the quotient map $\T(S,{\bf z}) \to \T(S)$ is a
holomorphic fibration \cite{bersfiber}.  Bers' Theorem holds in the more general situation that $S$ has finite type.
From this and an inductive argument, it follows that the fibration $\T(S,{\bf z}) \to \T(S)$ is holomorphic for any
finite set ${\bf z}$ (not just a single point), though we will not use this fact here.
\end{remark}

For each $X \in \H(S)$, we have the vector space of $1$--forms which are holomorphic with respect to $X$. This
determines a $g$--dimensional complex vector bundle over $\H(S)$ and we denote the bundle obtained from this by
removing the zero section by
$$\widetilde \Omega(S) = \{ (X,\omega) \, | \, X \in \H(S) \mbox{ and } \omega \mbox{ a holomorphic } 1\mbox{--form on } (S,X) \, \}.$$
We will refer to a point of $\widetilde \Omega(S)$ as $(X,\omega)$, or sometimes simply $\omega$ since the complex
structure $X$ is determined by the $1$--form $\omega$.  We let $\Zeros(\omega)$ denote the set of zeros of $\omega$.

We sometimes view $\omega \in \widetilde \Omega(S)$ as a \textit{translation structure} on $S$ (see
e.g.~\cite{eskinmasurzorich}).  This is a singular flat metric on $S$ with trivial holonomy and a preferred vertical
direction in each tangent space.  The singularities are isolated cone-type singularities occurring precisely at the
points of $\Zeros(\omega)$, and having cone angles in $2 \pi \mathbb Z$.  The metric (and notion of vertical) are
pulled back from $\mathbb C$ via \textit{natural coordinates} obtained by integrating $\omega$ over a sufficiently
small simply connected open neighborhood $U$ of a point $p_0$ in $S - \Zeros(\omega)$:
$$\zeta(p) = \int_{p_0}^p \omega.$$
We say that the natural coordinate $\zeta$ is based at $p_0$.  In the natural coordinates $\omega$ has the simple form
$\omega = d \zeta$.

The metric on $S$ associated to $\omega$ is locally $\text{CAT}(0)$.  Given $\alpha \in \C^0(S,{\bf z})$, there may
not be a geodesic representative in $S - \{{\bf z}\}$ as this surface is incomplete.  However, a sequence of representatives with lengths approaching
the infimum has a limit in $S$ (which may nontrivially intersect ${\bf z}$) by the Arzela--Ascoli Theorem.  This is a
geodesic, except possibly at points of ${\bf z}$ where incoming and outgoing geodesic segments can make an angle less
than $\pi$. We will refer to such curve as a geodesic representative for $\alpha$.

The right action of $\Diff^+(S)$ on $\H(S)$ naturally lifts to an action on $\widetilde \Omega(S)$.  This actions is equivalently the restriction of the action of $\Diff^+(S)$ on all $1$--forms
$$\omega \cdot \phi = \phi^{*}(\omega)$$
for $\phi \in \Diff^+(S)$ and $\omega \in \widetilde \Omega(S)$.  We consider two quotients
$$\Omega(S) = \widetilde \Omega(S) / \Diff_0(S) \quad \mbox{ and } \quad \Omega(S,{\bf z}) = \widetilde \Omega(S) / \Diff_0(S,{\bf z})$$
Equation (\ref{producteqn}) implies a product structure \begin{equation} \label{producteqn2} \Omega(S,{\bf z}) \cong
\Omega(S) \times \widetilde \Conf_n(S). \end{equation} Let
$$\pi: \Omega(S,{\bf z}) \to \Omega(S)$$
denote the projection.

\bigskip

Perhaps the most important point of what follows is the distinction between points of $\Omega(S,{\bf z})$ and of
$\Omega(S)$.   Given $\omega \in \widetilde \Omega(S)$ we will write $\hat \omega \in \Omega(S,{\bf z})$ and $\bar
\omega = \pi(\hat \omega) \in \Omega(S)$ for the associated points in the quotient spaces.

The right action of $\Diff^+(S)$ on $\widetilde \Omega(S)$ determines a left action in the usual way by defining
\begin{equation} \label{E:leftfromright}
\phi \cdot \omega = \phi^{-1 *}(\omega).
\end{equation}
This descends to left actions of $\Mod(S,{\bf z})$ and $\Mod(S)$ on
$\Omega(S,{\bf z})$ and $\Omega(S)$, respectively:
$$\hat \phi \cdot \hat \omega = \widehat{\phi^{-1 *} (\omega)} \quad \mbox{ and } \quad \bar \phi \cdot \bar \omega =
\overline{\phi^{-1 *} (\omega)}.$$

Let $\omega \in \widetilde \Omega(S)$ be any point.  We denote the fiber of $\pi$ over $\bar \omega$ by $F_{\bar
\omega} = \pi^{-1}(\bar \omega)$ and note that with respect to the product structure of (\ref{producteqn2}) we have
$$F_{\bar \omega} \cong \{ \hat \omega \} \times \widetilde \Conf_n(S).$$
From here we see the isomorphism (\ref{E:product1}) clearly; the action of $\B(S,{\bf z}) < \Mod(S,{\bf z})$ on
$\F_{\bar \omega}$ is by covering transformations.

\subsection{A neighborhood in the fiber} \label{movingsection}

We will frequently need to consider families of $1$--forms and not just their isotopy classes, and we use the
trivializations described in Section \ref{S:trivialization} to construct these. More precisely, consider any ${\bf
B}$--trivialization $f:S \times {\bf B} \to S$. Given $\omega \in \widetilde \Omega(S)$, $f$ determines a map we denote
$$f^\omega:{\bf B} \to \widetilde \Omega(S)$$
which is defined by
$$f^\omega({\bf b}) = f_{\bf b}^*\omega.$$
We can compose $f^\omega$ with the projections to both $\Omega(S,{\bf z})$ and $\Omega(S)$.  Since $f_{\bf b} \in
\Diff_0(S)$ for all ${\bf b} \in {\bf B}$, the latter map is simply the constant map with value $\bar \omega$. We are
primarily interested in the composition with the former projection, which we denote
$$\hat f^\omega: {\bf B} \to \Omega(S,{\bf z}).$$
The image of $\hat f^\omega$ lies in $F_{\bar \omega}$.  Since this is a local trivialization of the bundle
$$\Diff_0(S) \to \widetilde \Conf_n(S),$$
$\hat f^\omega$ maps onto a neighborhood of $\hat \omega = \hat f^\omega({\bf z})$ in $F_{\bar \omega}$.

\subsection{From $1$--forms to foliations} \label{S:1formfoliation}

An element $\omega \in \widetilde \Omega(S)$ determines a harmonic $1$--form, $\re(\omega)$, on $S$.  Let $\gamma_0$
denote any representative of a homotopy class $\gamma$ in $\Gamma(S)$ and $\omega \in \widetilde \Omega(S)$. Since
$\re(\omega)$ is harmonic, it is closed and hence the integral
$$\int_{\gamma_0} \re(\omega)$$
is independent of the choice of representative $\gamma_0$ of $\gamma$.

By definition of the left action of $\Diff^+(S)$ on $\widetilde \Omega(S)$, if $\phi \in \Diff^+(S)$ and $\omega \in
\widetilde \Omega(S)$, then
\begin{equation} \label{E:integral1}
\int_{\gamma_0} \re(\phi \cdot \omega) = \int_{\gamma_0} \phi^{-1 *}(\re(\omega)) = \int_{\phi^{-1}(\gamma_0)}
\re(\omega).
\end{equation}
If $\phi \in \Diff_0(S)$, then $\phi^{-1} \in \Diff_0(S)$, and so $\phi^{-1}(\gamma_0)$ also represents $\gamma$.
Therefore
$$\int_{\gamma_0} \re(\phi \cdot \omega) = \int_{\phi^{-1}(\gamma_0)} \re(\omega) = \int_{\gamma_0} \re(\omega).$$

It follows that we can well-define
$$\int_\gamma \re(\bar \omega) = \int_{\gamma_0} \re(\omega)$$
for any  $\gamma \in \Gamma(S)$ and $\bar \omega \in \Omega(S)$.

By the same reasoning, we can well-define
$$\int_\gamma \re(\hat \omega)$$
for any $\gamma \in \Gamma(S)$ or $\Gamma(S,z_i,z_j)$ and $\hat \omega \in \Omega(S,{\bf z})$, by picking arbitrarily
representatives of the relevant equivalence classes.  Furthermore, (\ref{E:integral1}) implies that the actions of
$\Mod(S,{\bf z})$ and $\Mod(S)$ satisfy
\begin{equation} \label{integration}
\int_{\bar \phi (\gamma)} \re(\bar \phi \cdot \bar \omega) = \int_\gamma \re(\bar \omega) \quad \mbox{ and } \quad
\int_{\hat \phi (\gamma)} \re(\hat \phi \cdot \hat \omega) = \int_\gamma \re(\hat \omega).
\end{equation}

A 1-form $\omega \in \widetilde \Omega(S)$ also determines a measured foliation on $S$.  The foliation is denoted
$\F(\re(\omega))$ as it is obtained by integrating the line field $\ker(\re(\omega))$. The measure $|\re(\omega)|$ is
obtained by integrating the absolute value of $\re(\omega)$.  Passing to the quotient by $\Diff_0(S,{\bf z})$ and
$\Diff_0(S)$ we obtain well-defined points $|\re(\hat \omega)| \in \MF(S,{\bf z})$ and $|\re(\bar \omega)| \in \MF(S)$,
respectively.

This determines a map
$$|\re|:\Omega(S,{\bf z}) \to \MF(S,{\bf z})$$
defined by
$$|\re|(\hat \omega) = |\re(\hat \omega)|.$$

\begin{lemma} \label{equivariant}
$|\re|$ is continuous and $\Mod(S,{\bf z})$--equivariant.
\end{lemma}
\begin{proof}
Continuity is well known (cf.~\cite{hubbardmasur}).  The idea is that given $\alpha \in \C^0(S,{\bf z})$, as $\omega
\in \widetilde \Omega(S)$ varies, the geodesic representatives vary continuously.  Since the geodesic representatives
realize $\I(\alpha,|\re(\omega)|)$, it easily follows that this quantity varies continuously, proving continuity of
$|\re|$.

To see the equivariance, we need only compare the various definitions.  Fixing $\hat \phi
\in \Mod(S,{\bf z})$ and $\hat \omega \in \Omega(S,{\bf z})$, we must show
$$\hat \phi \cdot |\re(\hat \omega)| = |\re(\hat \phi \cdot \hat \omega)|.$$
The action on $\MF(S,{\bf z})$ is determined by the action on $\C^0(S,{\bf z})$ via intersection numbers according to the equation
$$\I(\alpha,\hat \phi \cdot |\re(\hat \omega)|) = \I(\hat \phi^{-1}(\alpha),|\re(\hat \omega)|)$$
for every $\alpha \in \C^0(S,{\bf z})$. Therefore, we fix any $\alpha \in \C^0(S,{\bf z})$ and we must prove
$$\I(\hat \phi^{-1}(\alpha),|\re(\hat \omega)|) = \I(\alpha,|\re(\hat \phi \cdot \hat \omega)|).$$

Arbitrarily orienting $\alpha$ (i.e. coherently orienting all representatives of $\alpha$), and picking any
representative $\phi$ of $\hat \phi$ we obtain
\begin{align*}
 \I(\alpha,|\re(\hat \phi \cdot \hat \omega)|) &=  \I(\alpha,|\re(\widehat{\phi \cdot \omega})|)  = \inf_{\alpha_0 \in \alpha} \int_{\alpha_0} |\re(\phi \cdot \omega)|\\
  &=  \inf_{\alpha_0 \in \alpha} \int_{\alpha_0} |\phi^{-1*} (\re(\omega))|  =  \inf_{\alpha_0 \in \alpha} \int_{\alpha_0} \phi^{-1*} |\re(\omega)|\\
  &=  \inf_{\alpha_0 \in \alpha} \int_{\phi^{-1}(\alpha_0)} |\re(\omega)| = \inf_{\beta_0 \in \hat \phi^{-1}(\alpha)}  \int_{\beta_0} |\re(\omega)|\\
  &=  \I(\hat \phi^{-1}(\alpha),|\re(\hat \omega)|)
  \end{align*}
This proves equivariance and completes the proof of the lemma.
\end{proof}

\section{Periods and arationality}

Given $\hat \omega \in \Omega(S,{\bf z})$, we define the \textbf{periods of $\hat \omega$} by
$$\Per(\hat \omega) = \left\{ \int_\alpha \re(\hat \omega)  \, \big| \, \forall \, \alpha \in \Gamma(S) \right\}.$$
For each $i \neq j$ between $1$ and $n$, we define the \textbf{$ij$--relative periods
of $\hat \omega$} by
$$\Per_{ij}(\hat \omega) = \left\{ \int_\alpha \re(\hat \omega) \, \big| \, \forall \, \alpha \in \Gamma(S,z_i,z_j) \right\}.$$
We note that $\Per(\hat \omega)$ depends only on $\pi(\hat \omega) = \bar \omega$, whereas $\Per_{ij}(\hat \omega)$
actually depends on $\hat \omega$.

Our interest in the periods and relative periods comes from the following.
\begin{proposition}\label{arationalcriterion}
Suppose $\omega \in \widetilde \Omega(S)$, $|\re(\bar \omega)| \in \AF(S)$, and for every $i \neq j$ we have
$$\Per_{ij}(\hat \omega) \not \subset \Per(\hat \omega).$$
Then $|\re(\hat \omega)| \in \AF(S,{\bf z})$.
\end{proposition}
\begin{proof}
We will apply Lemma \ref{arationalprecriterion} and therefore need only check that for every $i \neq j$ the points
$z_i$ and $z_j$ are not connected by a leaf cycle of $\F(\ker(\re(\omega)))$.

Suppose on the contrary that there is a leaf cycle $\delta$ with endpoints $z_i$ and $z_j$. If $\epsilon$ is any path
from $z_i$ to $z_j$, then we can build a closed curve $\alpha = \delta \cup \epsilon$ by concatenating these two paths.
Because $\delta$ is a leaf cycle, the integral of $\re(\omega)$ over $\delta$ is zero and so
$$\int_\alpha \re(\omega) = \int_\delta \re(\omega)  + \int_\epsilon \re(\omega) = \int_\epsilon \re(\omega).$$

This implies $\Per_{ij}(\omega) \subset \Per(\omega)$ which is a contradiction.
\end{proof}

The following subspace of $\Omega$ is needed for technical reasons (see Section \ref{pathsinmfsection}).  Define the subspace $\widetilde \Omega_*(S,{\bf z}) \subset \widetilde \Omega(S)$ to be
$$\widetilde \Omega_*(S,{\bf z}) = \{ \omega \, | \, \Zeros(\omega) \cap {\bf z} = \emptyset \}$$
The group $\Diff_0(S,{\bf z})$ leaves $\widetilde \Omega_*(S,{\bf z})$ invariant and we let $\Omega_*(S,{\bf z})$ denote the
image in $\Omega(S,{\bf z})$.

\begin{lemma} \label{codimension2}
$\Omega_*(S,{\bf z})$ is path-connected and dense in $\Omega(S,{\bf z})$.
\end{lemma}
\begin{proof}
The space $\Omega(S,{\bf z})$ is the complement of the zero section of a complex vector bundle over the Teichm\"uller
space.  Since Teichm\"uller space is path-connected, so is $\Omega(S,{\bf z})$.  The space $\Omega_*(S,{\bf z})$ is the
complement of a subspace with real--codimension $2$:  the subspace $\{z_i \in \Zeros(\omega)\}$ is codimension 2 since
for any fixed $\omega \in \widetilde \Omega(S)$, $z_i$ and $\Zeros(\omega)$ are both zero-dimensional, and the zeros
(considered as a function from $\widetilde \Omega(S)$ to the $(2g-2)$--fold product $S \times ... \times S$, modulo the
action by the symmetric group) vary continuously.  It follows that $\Omega_*(S,{\bf z})$ is dense and path-connected.
\end{proof}

We now come to a more interesting subspace.  Define
$$\widetilde \Omega_\best(S,{\bf z}) = \left\{ \, \omega \in \widetilde \Omega_*(S,{\bf z}) \, \left| \, \begin{array}{l} |\re(\bar \omega)| \mbox{ is arational, and for every }\\ i \neq j \, , \, \, \Per_{ij}(\hat \omega) \not \subset \Per(\hat \omega) \end{array} \right. \, \right\}.$$
By construction $\widetilde \Omega_\best(S,{\bf z})$ is invariant by $\Diff_0(S,{\bf z})$ and we define
$\Omega_\best(S,{\bf z})$ to be the image in $\Omega_*(S,{\bf z})$.

\begin{proposition} \label{bestarational}
$|\re|(\Omega_\best(S,{\bf z})) \subset \AF(S,{\bf z})$.
\end{proposition}
\begin{proof}  This is immediate from Proposition \ref{arationalcriterion} and the definition of $\Omega_\best(S,{\bf z})$.
\end{proof}

In order for this subspace to be useful, we need the following.

\begin{proposition} \label{bestdense}
$\Omega_\best(S,{\bf z})$ is nonempty and dense in $\Omega(S,{\bf z})$.
\end{proposition}
\begin{proof}
First note that the set of $\omega \in \widetilde \Omega(S)$ for which $|\re(\bar \omega)|$ is arational is a dense
subset.  Indeed, for \textit{any} $\omega \in \widetilde \Omega(S)$, the set of $\theta$ for which $|\re(e^{i \theta}
\omega)|$ fails to be arational is countable---there are only countably many directions with a saddle connection; see
also \cite{KerMasSmi}.

We therefore fix $\omega \in \widetilde \Omega_*(S,{\bf z})$ such that $|\re(\bar \omega)|$ is arational and will prove
that the intersection $\Omega_\best(S,{\bf z}) \cap F_{\bar \omega}$ is dense in $F_{\bar \omega}$.  The proposition
will follow from this and Lemma \ref{codimension2}.

Density in $F_{\bar \omega}$ comes from basic genericity considerations: we will show that the relative periods vary
by a translation of $\mathbb R$ in a controlled way as one moves around within the fibers, while the periods do not change.
Since the sets of periods and relative periods are countable, this will easily
imply the result. We now explain this more precisely.

We will consider a ${\bf B}$--trivialization, $f:S \times {\bf B} \to S$. Let $f^\omega: {\bf B} \to \widetilde
\Omega(S)$ and $\hat f^\omega: {\bf B} \to \Omega(S,{\bf z})$ be the associated maps as in Section \ref{movingsection}.

We choose the specific ${\bf B} = B_1 \times ... \times B_n$ and ${\bf U} = U_1 \times ... \times U_n$ so that for each
$j = 1,...,n$, the natural coordinate $\zeta_j$ based at $z_j$ is defined on $U_j$ (see Section \ref{abeliansection}).
Moreover, we require that $\zeta_j$ maps $B_j$ diffeomorphically onto a square in $\mathbb C$. That is, there exists
$\epsilon > 0$ so that
$$\zeta_j(B_j) = (-\epsilon,\epsilon)^2 =  \{ x+ iy \in \mathbb C \, | \, x,y \in (-\epsilon,\epsilon) \}.$$
Observe that $U_j \cap \Zeros(\omega) = \emptyset$.  Since $\Zeros(f^\omega({\bf b})) = \Zeros(f_{\bf b}^*(\omega))$,
$f_{\bf b}$ is the identity outside $\cup_i U_i$ and $f_{\bf b}(z_i) = b_i \in B_i \subset U_i$, it follows that $\hat
f^\omega({\bf B}) \subset \Omega_*(S,{\bf z})$.

\begin{claim}
There exists a dense subset ${\bf E} \subset {\bf B}$ so that $\hat f^\omega({\bf E}) \subset \Omega_\best(S,{\bf z})$.
Equivalently, $f_{\bf b}^*(\omega) \in \widetilde \Omega_\best(S,{\bf z})$ for all ${\bf b} \in {\bf E}$.
\end{claim}

Since $\hat f^{\omega}({\bf B})$ is a neighborhood of $\hat \omega$ in $F_{\bar \omega}$, this claim implies that there
exists a point of $\Omega_\best(S,{\bf z})$ arbitrarily close to $\hat \omega$. Since $\hat \omega$ was an arbitrary
point of a dense subset, this will prove the proposition.

\begin{proof}[Proof of Claim.]
Let $\gamma:[0,1] \to S$ be any path from $z_i$ to $z_j$, representing an element of $\Gamma(S,z_i,z_j)$. The following
says that the change in $\gamma$--period from $\omega$ to $f_{\bf b}^*(\omega)$ is independent of $\gamma$, and is
given by a simple function defined on ${\bf B}$.
\begin{subclaim}
$$\int_\gamma \re(f_{\bf b}^*\omega) - \int_\gamma \re(\omega) = \re(\zeta_j(b_j)) - \re(\zeta_i(b_i))$$
\end{subclaim}

\begin{proof}[Proof of Subclaim.]
For any ${\bf b} \in {\bf B}$ fix a path ${\bf \sigma}:[0,1] \to {\bf B}$ going from ${\bf z}$ to ${\bf b}$, writing
${\bf \sigma}(t) = (\sigma_1(t),...,\sigma_n(t))$. This determines a map
$$H:[0,1] \times [0,1] \to S$$
by $H(t,u) = f_{{\bf \sigma}(u)}(\gamma(t))$.

The restriction of $H$ to the boundary of $[0,1] \times [0,1]$ determines four paths.  The bottom path is $H(t,0) =
\gamma(t)$.  The top path is $H(t,1) = f_{\sigma(1)}(\gamma(t)) = f_{\bf b}(\gamma(t))$. The left side path, oriented
up, is $H(0,u) = f_{\sigma(u)}(\gamma(0)) = f_{\sigma(u)}(z_i) = \sigma_i(u)$ and the right side path, also oriented
up, is $H(1,u) = f_{\sigma(u)}(\gamma(1)) = f_{\sigma(u)}(z_j) = \sigma_j(u)$.  Because $\re(\omega)$ is closed, the
integral over the boundary of $[0,1] \times [0,1]$ of $H^*(\re(\omega))$ is zero and so
$$0 = \int_{H(\partial ([0,1] \times [0,1]))} \re(\omega) = \int_{f_{\bf b}(\gamma)} \re(\omega) + \int_{\sigma_i} \re(\omega) - \int_{\sigma_j} \re(\omega) - \int_\gamma \re(\omega).$$
Since $\sigma_i$ connects $z_i$ to $b_i$ within $B_i$ and $\sigma_j$ connects $z_j$ to $b_j$ within $B_j$ we have
$$\int_{\sigma_i} \re(\omega) = \re(\zeta_i(b_i)) \quad \mbox{ and } \quad \int_{\sigma_j} \re(\omega) = \re(\zeta_j(b_j)).$$
Combining this with the previous equation and the descriptions of the four paths involved in that equation we obtain
$$\int_\gamma \re(f_{\bf b}^*\omega) = \int_{f_{\bf b}(\gamma)} \re(\omega) =
\int_\gamma \re(\omega) + \re(\zeta_j(b_j)) - \re(\zeta_i(b_i))$$ and this proves the subclaim.
\end{proof}

We now see that
$$\Per_{ij}(\hat f^\omega({\bf b})) = \Per_{ij}(\widehat{f_{\bf b}^*(\omega)}) = \Per_{ij}(\widehat \omega) + \re(\zeta_j(b_j)) - \re(\zeta_i(b_i)).$$
That is, the subsets of $\mathbb R$, $\Per_{ij}(\hat \omega)$ and $\Per_{ij}(\hat f^\omega({\bf b}))$ differ exactly by
a translation by $\re(\zeta_j(b_j)) - \re(\zeta_i(b_i))$.  Since the set of periods and relative periods are all
countable sets and since $\Per(\hat f^\omega({\bf b})) = \Per(\hat \omega)$ it follows that for almost all ${\bf b}$ we
have
$$\Per_{ij}(\hat f^\omega({\bf b})) \cap \Per(\hat f^\omega({\bf b})) = \emptyset$$
In particular, setting
$${\bf E} = \left\{ {\bf b} \, \big| \, \Per_{ij}(\hat f^\omega({\bf b})) \cap \Per(\hat f^\omega({\bf b})) = \emptyset \right\}$$
we have found the required set.
\end{proof}

This completes the proof of the proposition.
\end{proof}

\begin{corollary} \label{denseinmf}
$|\re|(\Omega_\best(S,{\bf z}))$ is dense in $\AF(S,{\bf z})$.
\end{corollary}
\begin{proof}
Since $\Omega_\best(S,{\bf z})$ is invariant by $\Mod(S,{\bf z})$, Lemma \ref{equivariant} implies that
$|\re|(\Omega_\best(S,{\bf z}))$ is also $\Mod(S,{\bf z})$--invariant.  Now, $\Mod(S,{\bf z})$ acts minimally on
$\PMF(S,{\bf z})$---see Theorem 6.7 of \cite{FLP}---and so it follows that the image of $|\re|(\Omega_\best(S,{\bf
z}))$ in $\PMF(S,{\bf z})$ is dense.  This implies that the same is true of $|\re|(\Omega_\best(S,{\bf z}))$ in
$\MF(S,{\bf z})$, and so also in $\AF(S,{\bf z})$.
\end{proof}

\section{Paths in $\AF(S,{\bf z})$} \label{pathsinmfsection}

In this section, we prove the key ingredient which produces an abundance of paths in $\AF(S,{\bf z})$.
\begin{theorem} \label{connectingneighborhood}
There is an open cover of $\mathcal U$ of $\Omega_\zero(S,{\bf z})$ with the property that for any $U \in \mathcal U$
and any $\hat \omega,\hat \eta \in U \cap \Omega_\best(S,{\bf z})$, there is a path in $\AF(S,{\bf z})$ connecting
$|\re(\hat \omega)|$ and $|\re(\hat \eta))|$.
\end{theorem}

\subsection{Twisting pairs} \label{twistpairsection}

Fix a point $\omega \in \widetilde \Omega_*(S,{\bf z})$.  We say that a pair of simple closed curves $\alpha$ and
$\beta$ on $S$ are a \textbf{twisting pair} for $\omega$ if
\begin{itemize}
\item $\alpha$ and $\beta$ meet transversely and minimally,
\item $\alpha$ and $\beta$ fill $S$,
\item ${\bf z} \subset \alpha \cap \beta$, and
\item $\alpha$ and $\beta$ are transverse to $\F(\re(\omega))$, and $\alpha \cap \Zeros(\omega) = \beta \cap \Zeros(\omega) = \emptyset$.
\end{itemize}

\begin{lemma} \label{twistpairsexist}
For any $\omega \in \widetilde \Omega_*(S,{\bf z})$, there is a twisting pair $\alpha,\beta$ for $\omega$.
\end{lemma}
\begin{proof} Pick two distinct points $e^{i \theta_1},e^{i \theta_2} \in S^1 \subset \mathbb C$ neither of which is equal to $1$ and
let $\F_j = \F(\re(e^{i \theta_j} \omega))$ and $\mu_j = |\re(e^{i \theta_j} \omega)|$ for $j = 1,2$.  We choose $e^{i
\theta_1},e^{i \theta_2}$ so that $(\F_1,\mu_1)$ and $(\F_2,\mu_2)$ are uniquely ergodic and arational. According to
\cite{KerMasSmi}, this is true for almost every $e^{i \theta} \in S^1$. We also assume, as we may, that the leaves
through $z_1$ do not pass through $\Zeros(\omega) = \Zeros(e^{i \theta_j} \omega)$, for $j = 1,2$.

\begin{figure}[htb]
\centerline{}
\centerline{}
\begin{center}
\ \psfig{file=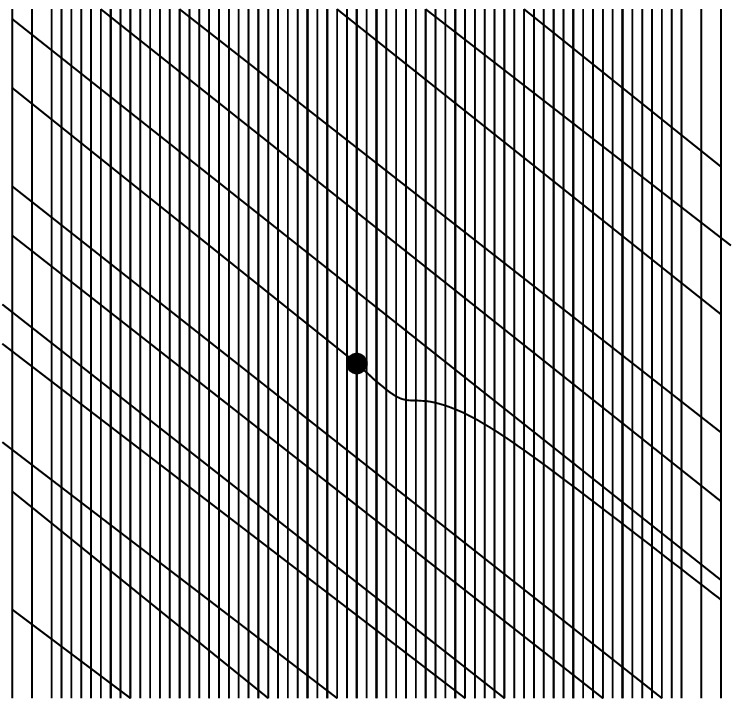,height=1.5truein} \caption{Closing up to a simple closed curve through $z_1$.}
\label{closeup}
\end{center}
\end{figure}

For each $j = 1,2$, we construct a sequence of simple closed curves $\{\gamma^j_k\}_{k=1}^\infty$ which approximate
leaves of $\F_j$.  Start with a ray in the leaf of $\F_j$ emanating from $z_1$. Take an initial segment of length at
least $k$ which comes sufficiently close to $z_1$ so that it can be smoothly closed up to a simple closed curve
$\gamma^j_k$ transverse to $\F(\re(\omega))$; see Figure \ref{closeup}. For this, one can work in the natural
coordinate $\zeta_1$ based at $z_1$ for $\omega$ (as described in Section \ref{abeliansection}) and appeal to the fact
that the ray is dense in $S$ by arationality of $\F_j$.

Given any $\epsilon > 0$, we can also assume that the curve $\gamma^j_k$ makes an angle at most $\epsilon$ with $\F_j$
at every point, provided $k$ is sufficiently large.  In particular, if $\epsilon$ is sufficiently small, it follows
that $\gamma^1_k$ and $\gamma^2_k$ will intersect transversely and minimally.

Observe that by construction, $\I(\gamma^j_k,\mu_j) \to 0$ as $k \to \infty$, and hence by unique ergodicity and
arationality, $\gamma^j_k \to \mathbb P (\mu_j)$ in $\PMF(S)$ for each $j = 1,2$. So by taking $k$ even larger if
necessary, we may assume that $\gamma^1_k$ and $\gamma^2_k$ fill $S$.

Finally, because both rays are dense and transverse to each other, their points of intersection are dense.  Hence by
taking $k$ larger still, we can guarantee that the set of intersection points of $\gamma^1_k$ and $\gamma^2_k$ are
$\epsilon$--dense.  In particular, each $z_i$ is within $\epsilon$ of a point of intersection of $\gamma^1_k$ and
$\gamma^2_k$.  For sufficiently small $\epsilon$ (again, working in a natural coordinate), we can perturb $\gamma^1_k$
and $\gamma^2_k$ to simple closed curves $\alpha$ and $\beta$ which are a twisting pair for $\omega$.  See Figure
\ref{perturb}.
\begin{figure}[htb]
\centerline{}
\centerline{}
\begin{center}
\ \psfig{file=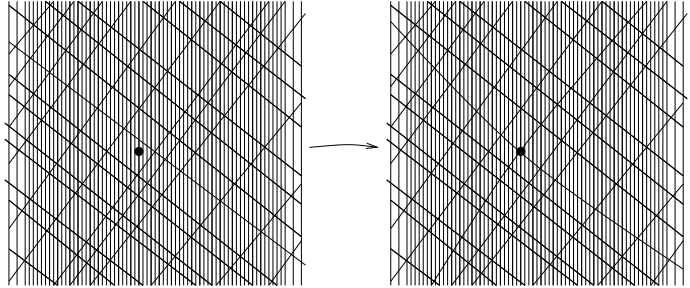,height=1.5truein}
\caption{Perturbing $\gamma^1_k$ and $\gamma^2_k$ to $\alpha$ and $\beta$ passing through $z_i$.}
\label{perturb}
\end{center}
\end{figure}
\end{proof}

\subsection{The group of a twisting pair.}

Now let $\alpha,\beta$ be a twisting pair for $\omega \in \widetilde \Omega_*(S)$.  Our first goal is to define
isotopies, $D_{\alpha,t}$ and $D_{\beta,t}$ supported on annular neighborhoods of $\alpha$ and $\beta$, respectively,
which push the set ${\bf z}$ once around $\alpha$ and $\beta$, respectively, at a constant speed as measured with
respect to $\re(\omega)$.  We will use these isotopies to construct paths in $\Omega_\best(S,{\bf z})$ which will be
used in the proof of Theorem \ref{connectingneighborhood}. We now describe this in more detail.

Let $\epsilon > 0$ be such that the $\epsilon$--neighborhood $N_\epsilon(\alpha)$ is an annulus, and the foliation
$\F(\re(\omega))$ restricted to this annulus is the product foliation $N_\epsilon(\alpha) \cong S^1 \times [0,1]$. More
precisely, we have a diffeomorphism
$$f_\alpha:N_\epsilon(\alpha) \to S^1 \times [0,1]$$
which we assume takes $\alpha$ to $S^1 \times \{1/2\}$. Moreover, we can choose $f_\alpha$ so that
$$f_\alpha^*(ds) = \frac{1}{c} \re(\omega)$$
where $c = \I(\alpha,|\re(\omega)|)$ and $ds$ is the $1$--form coming from the factor $S^1 = \mathbb R/\mathbb Z$ where it is defined by the standard
coordinate $s$ on $\mathbb R$.
Let $D_{\alpha,t}:S \to S$, $t \in [0,1]$ be an isotopy supported on $N_\epsilon(\alpha)$ defined as follows.

Let $\psi:[0,1] \to [0,1]$ be a smooth function identically zero in a neighborhood of $0$ and $1$, and equal to $1$ at
$1/2$.  Define
$$D_{\alpha,t}^0:S^1 \times [0,1] \to S^1 \times [0,1]$$
by
$$D_{\alpha,t}^0(s,x) = (s + t \psi(x),x).$$
Then define $D_{\alpha,t}$ to be the identity outside $N_\epsilon(\alpha)$ and equal to
$$f_\alpha^{-1} D_{\alpha,t}^0 f_\alpha$$
on $N_\epsilon(\alpha)$.

Likewise, we can define
$$D_{\beta,t}:S \to S$$
using $\beta$ in place of $\alpha$.
\begin{proposition} \label{P:beststillbest}
If $\omega \in \widetilde \Omega_\best(S,{\bf z})$, then $D_{\alpha,t}^*(\omega),D_{\beta,t}^*(\omega) \in \widetilde
\Omega_\best(S,{\bf z})$ for all $t \in [0,1]$.
\end{proposition}
\begin{proof}
Observe that $\delta_i(t) = D_{\alpha,t}(z_i)$, for $t \in [0,1]$ is a parameterization of $\alpha$ starting at $z_i$
with
\begin{equation} \label{changeinperiods} \int_{\delta_i([0,t])} \re(\omega) = t \int_\alpha
\re(\omega).\end{equation} A computation similar to the one done in the proof of Proposition \ref{bestdense} tells us
that for any $\gamma \in \Gamma(S,z_i,z_j)$, we have
$$\int_\gamma D_{\alpha,t}^*(\re(\omega)) - \int_\gamma \re(\omega) =
\int_{\delta_j([0,t])} \re(\omega) - \int_{\delta_i([0,t])} \re(\omega).$$ According to (\ref{changeinperiods}), this
becomes
$$\int_\gamma D_{\alpha,t}^*(\re(\omega)) - \int_\gamma \re(\omega) = t\int_{\alpha} \re(\omega) -
t\int_{\alpha} \re(\omega) = 0.$$ It follows that
\begin{equation} \label{E:nochangeinperalpha} \Per_{ij}(\omega) = \Per_{ij}(D_{\alpha,t}^*(\omega)), \quad \forall t \in [0,1].\end{equation}
Similarly, for $D_{\beta,t}$ we obtain
\begin{equation} \label{E:nochangeinperbeta} \Per_{ij}(\omega) = \Per_{ij}(D_{\beta,t}^*(\omega)), \quad \forall t \in [0,1].\end{equation}

Finally, observe that for all $t \in [0,1]$
\[|\re(\overline{D_{\alpha,t}^*(\omega)})| = |\re(\bar \omega)| = |\re(\overline{D_{\beta,t}^*(\omega)})| .\]
From this, equations \eqref{E:nochangeinperalpha} and \eqref{E:nochangeinperbeta}, and the definition of $\widetilde
\Omega_\best(S,{\bf z})$, the proposition follows.
\end{proof}

The diffeomorphisms $D_\alpha = D_{\alpha,1}$ and $D_\beta = D_{\beta,1}$ are in $\Diff_{0,{\bf z}}(S,{\bf z})$.  In
fact $D_\alpha$ can be alternatively described as a Dehn twist in one component of $\partial N_{\epsilon}(\alpha)$
composed with an inverse Dehn twist in the other, and similarly for $D_\beta$.  As usual, we let $\hat D_\alpha ,\hat
D_\beta \in \B(S,{\bf z}) < \Mod(S,{\bf z})$ denote the associated mapping classes.

As described in Section \ref{abeliansection} (\ref{E:leftfromright}), the left action is described by the equations
$$D_\alpha^{-1} \cdot \omega = D_\alpha^*(\omega)$$
and
$$D_\beta^{-1} \cdot \omega = D_\beta^*(\omega).$$

Let $\G(\hat D_\alpha,\hat D_\beta)$ be the Cayley graph of $\langle \hat D_\alpha, \hat D_\beta \rangle < \B(S,{\bf
z})$ with respect to the generators $\hat D_\alpha,\hat D_\beta$.
\begin{proposition}
Suppose $\alpha,\beta$ is a twisting pair for $\omega$.   If $\hat \omega \in \Omega_\best(S,{\bf z})$, then there is a
$\langle \hat D_\alpha, \hat D_\beta \rangle$--equivariant continuous map
$$P_{\hat \omega}: \G(\hat D_\alpha,\hat D_\beta) \to \Omega_\best(S,{\bf z})$$
sending $\id \in \langle \hat D_\alpha, \hat D_\beta \rangle$ to $\hat \omega$.
\end{proposition}
\begin{proof}
Proposition \ref{P:beststillbest} implies that $t \mapsto D_{\alpha,t}^*(\omega)$ is a path from $\omega$ to
$D_\alpha^*(\omega)$ in $\widetilde \Omega_\best(S,{\bf z})$.  Projecting down to $\Omega_\best(S,{\bf z})$, we obtain
a path from $\hat \omega$ to $\hat D_\alpha^{-1} \cdot \hat \omega$.  Likewise, we obtain a path from $\hat \omega$ to
$\hat D_\beta^{-1} \cdot \hat \omega$.

We now define
$$P_{\hat \omega}: \G( \hat D_\alpha,\hat D_\beta) \to \Omega_\best(S,{\bf z}).$$
We do this first on the edge from $\id$ to $\hat D_\alpha^{-1}$ by sending it to the path from $\hat \omega$ to $\hat
D_\alpha^{-1} \cdot \hat \omega$ and sending the edge from $\id$ to $\hat D_\beta^{-1}$ to the path from $\hat \omega$
to $\hat D_\beta^{-1} \cdot \hat \omega$.  There is now a unique way to extend this to a $\langle \hat D_\alpha, \hat
D_\beta \rangle$--equivariant continuous map.
\end{proof}

\begin{lemma} \label{pseudoexists}
If $\alpha,\beta$ is a twisting pair for $\omega$, then $\langle \hat D_\alpha, \hat D_\beta \rangle < \B(S,{\bf z})$
contains a pseudo-Anosov mapping class $\psi$.
\end{lemma}
\begin{proof}
Since $\alpha,\beta$ is a twisting pair, they fill $S$.  This implies that $\partial N_\epsilon(\alpha),\partial
N_\epsilon(\beta)$ fill $(S,{\bf z})$ and so for any $\delta \in \C^0(S,{\bf z})$, $\I(\delta,\partial
N_\epsilon(\alpha)) \neq 0$ or $\I(\delta,\partial N_\epsilon(\beta)) \neq 0$. Since $\hat D_\alpha$ and $\hat D_\beta$
are multitwists in $N_\epsilon(\alpha)$ and $N_\epsilon(\beta)$, respectively, it follows that $\hat D_\alpha^k(\delta)
\neq \delta$ or $\hat D_\beta^k(\delta) \neq \delta$ for all $k \neq 0$.

In particular, there is no curve  $\delta$ with a finite $\langle \hat D_\alpha,\hat D_\beta \rangle$--orbit. That is,
a finite index pure subgroup of $\langle \hat D_\alpha,\hat D_\beta \rangle$ is an \textit{irreducible} subgroup of
$\B(S,{\bf z}) < \Mod(S,{\bf z})$. According to a theorem of Ivanov---Theorem 5.9 of \cite{ivanov}---there exists a
pseudo-Anosov mapping class $\psi \in \langle \hat D_\alpha,\hat D_\beta \rangle$.
\end{proof}

Given $\omega \in \widetilde \Omega_\best(S)$ and a twisting pair $\alpha,\beta$ for $\omega$, we let $\langle \hat
D_\alpha,\hat D_\beta \rangle$ be the associated group, and $\psi$ a pseudo-Anosov element guaranteed by  Lemma
\ref{pseudoexists}. We let $\mu_s$ and $\mu_u$ denote stable and unstable foliations for $\psi$, respectively (unique
up to scalar multiple).

\begin{lemma} \label{pathtostable}
There exists a path connecting $|\re(\hat \omega)|$ to $\mu_s$ in $\AF(S,{\bf z})$.
\end{lemma}
\begin{proof}
It is more convenient to work in the space $\PMF(S,{\bf z})$ and we write $\mathbb P: \MF(S,{\bf z}) \to \PMF(S,{\bf
z})$ for the quotient map.  Since the fibers of $\Phi$ are homeomorphic to $\mathbb R_+$ and since there exists a
section of $\mathbb P$, a path between projective classes $\mathbb P(|\re(\hat \omega)|)$ and $\mathbb P(\mu_s)$ easily
implies the existence of a path between any representatives $|\re(\hat \omega)|$ and $\mu_s$.  The advantage to working
in $\PMF(S,{\bf z})$ is that we can appeal to the dynamics as described in Section \ref{foliationsection}.

Consider the path in $\G(\hat D_\alpha, \hat D_\beta)$ given by $[\rm{Id},\psi] \cup [\psi,\psi^2] \cup [\psi^2,\psi^3]
\cup ... $, where $[\rm{Id},\psi]$ is a geodesic from $\rm{Id}$ to $\psi$, and $[\psi^k,\psi^{k+1}]$ is the image of
this geodesic under $\psi^k$. We can parameterize this
$$f: [0,1) \to [\rm{Id},\psi] \cup [\psi,\psi^2] \cup ...$$
sending $[0,\frac{1}{2}]$ linearly onto the first segment, $[\frac{1}{2},\frac{3}{4}]$ linearly onto the second
segment, and so on.

\begin{claim} $h = \mathbb P \circ |\re| \circ P_{\hat \omega} \circ f:[0,1) \to \PMF(S,{\bf z})$ extends to a continuous
map defined on $[0,1]$ by setting $h(1) = \mathbb P(\mu_s)$.\end{claim}
\begin{proof}[Proof of claim]
Let $\{x_k\}$ be any sequence in $[0,1)$ tending to $1$ and we show that $h(x_k) \to \mathbb P(\mu_s)$.  Lemma
\ref{nonorientstable} implies that $\mathbb P(\mu_u) \not \in h([0,1))$.  Since
$$h([0,1)) = \bigcup_{j=0}^\infty \psi^j(h([0,1/2]))$$
it follows that for each $k$, there is a $j(k)$ so that $f(x_k) \in [\psi^{j(k)},\psi^{j(k)+1}]$.  Since $x_k \to 1$,
it must be that $j(k) \to \infty$ as $k \to \infty$.

Now let $V$ be any neighborhood of $\mathbb P(\mu_s)$ in $\PMF(S,{\bf z}) - \{ \mathbb P(\mu_u)\}$. Since
$h([0,\frac{1}{2}])$ is compact, there exists $J > 0$ so that for all $j \geq J$, $\psi^j(h([0,\frac{1}{2}])) \subset
V$.  By the previous paragraph there exists $K > 0$ so that for all $k \geq K$, $j(k) \geq J$, which implies
\[ h(x_k) = \mathbb P \circ |\re| \circ P_{\hat \omega} \circ f(x_k) \in \mathbb P \circ |\re| \circ P_{\hat \omega} \left( [\psi^{j(k)},\psi^{j(k)+1}]
\right). \] Since we also have
\begin{align*}
\mathbb P \circ |\re| \circ P_{\hat \omega} \left( [\psi^{j(k)},\psi^{j(k)+1}] \right) &= \mathbb P \circ |\re| \circ P_{\hat \omega} \left( \psi^{j(k)}([\rm{Id},\psi]) \right)\\
 &= \psi^{j(k)}\left( \mathbb P \circ |\re| \circ P_{\hat \omega} \circ f \left( \left[0,1/2\right] \right) \right)\\
 &= \psi^{j(k)}\left( h\left(\left[0,1/2\right] \right) \right) \subset V. \end{align*}
it follows that $h(x_k) \in V$. This completes the proof of the claim.
\end{proof}

The image of $h:[0,1] \to \PMF(S,{\bf z})$ is contained in $\PAF(S,{\bf z})$, and connects $\mathbb P(|\re(\hat
\omega)|)$ to $\mathbb P(\mu_s)$, as required.
\end{proof}

\subsection{The open cover}

The definition of a twisting pair for $\omega \in \widetilde \Omega_*(S)$ had only one condition which involved
$\omega$. Namely, we required that each of $\alpha,\beta$ be transverse to $\F(\re(\omega))$.  It is not surprising
then that a twisting pair for $\omega$ is also a twisting pair for elements of $\widetilde \Omega_*(S)$ which are
sufficiently close to $\omega$.

\begin{lemma} \label{pairsnearby}
Given $\omega \in \widetilde \Omega_*(S)$ and a twisting pair $\alpha,\beta$ for $\omega$ there is a neighborhood $U'$
of $\omega$ so that for all $\eta \in U'$, $\alpha,\beta$ is a twisting pair for $\eta$.
\end{lemma}
\begin{proof}
By definition, the underlying foliation $\F(\re(\omega))$ is obtained by integrating $\ker(\re(\omega))$.  Thus, the
condition that a curve $\gamma$ be transverse to $\F(\re(\omega))$ is equivalent to requiring that $\re(\omega)$
restricted to $\gamma$ be nonvanishing.  Perturbing the $1$--form $\re(\omega)$ slightly preserves the property that it
is nonvanishing on $\gamma$ since $\gamma$ is compact.  Therefore, there is a neighborhood $U'$ of $\omega$ in
$\widetilde \Omega_*(S)$ so that if $\eta \in U'$, then the restriction of $\re(\eta)$ to both $\alpha$ and $\beta$ is
nonvanishing. It follows that $\alpha,\beta$ is a twisting pair for $\eta$, as required.
\end{proof}

We can now give the
\begin{proof}[Proof of Theorem \ref{connectingneighborhood}]
Fix $\nu \in \widetilde \Omega_*(S)$, and let $\alpha,\beta$ is a twisting pair for $\nu$.  Lemma \ref{pairsnearby}
provides a neighborhood $U'$ so that for all $\omega \in U'$, $\alpha,\beta$ is also a twisting pair for $\omega$. We
let $\hat \psi$ be a pseudo-Anosov mapping class in $\langle \hat D_\alpha,\hat D_\beta \rangle$ as in Lemma
\ref{pseudoexists} and $\mu_s$ its stable foliation.

It follows from the discussions in Sections \ref{S:trivialization} and \ref{abeliansection} that we can locally find a
continuous section of $\widetilde \Omega_*(S) \to \Omega_*(S,{\bf z})$. In particular, there is a neighborhood $U$ of
$\hat \nu$ and a continuous section $\sigma: U \to \widetilde \Omega_*(S)$ with $\sigma(U) \subset U'$. Now, given
$\hat \omega , \hat \eta \in U \cap \Omega_\best(S,{\bf z})$, since $\alpha,\beta$ is a twisting pair for both $\omega$
and $\eta$, Lemma \ref{pathtostable} guarantees paths from $|\re(\hat \omega)|$ to $\mu_s$ and from $|\re(\hat \eta)|$
to $\mu_s$ in $\AF(S,{\bf z})$. Therefore, we can connect $|\re(\hat \omega)|$ and $|\re(\hat \eta)|$ by a path in
$\AF(S,{\bf z})$.  Since $\nu$ was an arbitrary point of $\widetilde \Omega_*(S)$, we have constructed the open cover.
\end{proof}

\subsection{The main theorem for ${\bf z} \neq \emptyset$}

We now put all the ingredients together to prove the main theorem for surfaces of genus at least $2$ and nonempty
marked point set.
\begin{proof}[Proof of Theorem \ref{mainarational}, ${\bf z} \neq \emptyset$.]
Corollary \ref{denseinmf} implies that $|\re|(\Omega_\best(S,{\bf z}))$ is dense in $\AF(S,{\bf z})$. Therefore, to
prove that $\AF(S,{\bf z})$ is connected, we show that for any two points $\hat \omega, \hat \eta \in
\Omega_\best(S,{\bf z})$, there is a path connecting $|\re|(\hat \omega) = |\re(\hat \omega)|$ to $|\re|(\hat \eta) =
|\re(\hat \eta)|$ in $\AF(S,{\bf z})$.  This produces a dense path-connected subset of $\AF(S,{\bf z})$, and so
$\AF(S,{\bf z})$ is connected.

Let $\hat \omega,\hat \eta \in \Omega_\best(S,{\bf z})$ be any two points.  According to Lemma \ref{codimension2} there
is a path $\delta$ in $\Omega_*(S,{\bf z})$ connecting $\hat \omega$ to $\hat \eta$.   Because $\delta$ is compact, the
cover $\mathcal U$ restricted to $\delta$ has a finite subcover.  From this we can produce a finite set $U_0,...,U_k
\in \mathcal U$ such that $\hat \omega \in U_0$, $\hat \eta \in U_k$, and $U_j \cap U_{j+1} \neq \emptyset$ for all $j
= 0,...,k-1$.  For each $j = 0,...,k-1$, appealing to Proposition \ref{bestdense} we can therefore find some element in
the intersection
$$\hat \omega_j \in U_j \cap U_{j+1} \cap \Omega_\best(S,{\bf z})$$
By Theorem \ref{connectingneighborhood}, for every $j = 0,...,k-2$ since $\hat \omega_j$ and $\hat \omega_{j+1}$ are in
$U_{j+1} \cap \Omega_\best(S,{\bf z})$, there is a path in $\AF(S,{\bf z})$ connecting $|\re(\hat \omega_j)|$ and
$|\re(\hat \omega_{j+1})|$.  Likewise, there is a path connecting $|\re(\hat \omega)|$ to $|\re(\hat \omega_0)|$ and
$|\re(\hat \eta)|$ to $|\re(\hat \omega_{k-1})|$.  Concatenating these paths we obtain a path connecting $|\re (\hat \omega)|$
to $|\re(\hat \eta)|$, as required.
\end{proof}

\subsection{Cut points: ${\bf z} \neq \emptyset$} \label{cutpointsection1}

For the applications to Kleinian groups, we need to see that $\AF(S,{\bf z})/\hspace{-.1cm} \sim$ has no cut points,
meaning no points whose removal disconnects.  Together with the proof in the previous section, the following proves
Theorem \ref{endinglams} for the case ${\bf z} \neq \emptyset$.

\begin{theorem} \label{nocut1}
If ${\bf z} \neq \emptyset$ then $\AF(S,{\bf z})/\hspace{-.1cm} \sim$ has no cut points.
\end{theorem}
In what follows, we let $[\mu]$ denote the equivalence class in $\AF(S,{\bf z})/\hspace{-.1cm} \sim$ of the measured
foliation $\mu \in \AF(S,{\bf z})$.
\begin{proof}
The dense path connected subset of $\AF(S,{\bf z})$ descends to a dense path connected subset $W$ in $\AF(S,{\bf
z})/\hspace{-.1cm} \sim$. To prove the theorem, it suffices to verify that given any point of $\AF(S,{\bf
z})/\hspace{-.1cm} \sim$, there is a dense path connected subset $W_0$ of the complement.  Let $|\mu| \in \AF(S,{\bf
z})/\hspace{-.1cm} \sim$ be an arbitrary point.  If $|\mu| \not \in W$, then we may take $W = W_0$, so we assume that
$|\mu| \in W$.

Points of $W$ are of two types: (1) the stable foliations of the pseudo-Anosov mapping classes coming from Lemma
\ref{pathtostable} and (2) the points in the image of $|\re|(\Omega_\best(S,{\bf z}))$.  It follows from Lemma
\ref{nonorientstable} that these two subsets of $W$ are disjoint.

Suppose first that $\mu = |\re|(\hat \omega) \in |\re|(\Omega_\best(S,{\bf z}))$, so we are in case (2).  We can define
\[ \Omega'_\best(S,{\bf z}) = \left\{ \, \hat \eta \in \Omega_\best(S,{\bf z}) \, \, \mid \, \, [|\re|(\bar \eta) ] \neq [|\re|(\bar \omega)] = [\mu] \, \right\} \]
The space $\Omega'_\best(S,{\bf z})$ also has the property that it is dense in $\Omega(S,{\bf z})$ and has dense image
in $\AF(S,{\bf z})$.  This set can be used in place of $\Omega_\best(S,{\bf z})$ to build a path connected dense set,
with image $W_0$ in $\AF(S,{\bf z})/\hspace{-.1cm} \sim$.  By construction $[\mu] \not \in W_0$, and so it provides the
required connected dense subset.

Now we assume that $\mu=\mu_s$ is the stable foliation for a pseudo-Anosov $\psi$ as found in Lemma \ref{pathtostable},
and that $[\mu] \in W$.  Such a foliation $[\mu_s] \in W$ must arise from the open cover $\mathcal U$ constructed in
the proof of Theorem \ref{connectingneighborhood} as follows.  For some $U \in \mathcal U$, we showed that
\[ |\re|(\Omega_\best(S,{\bf z}) \cap U) \cup \{\mu_s\} \]
is path connected.  Since the unstable foliation $\mu_u$ for $\psi$ is the stable foliation for $\psi^{-1}$, it follows
that the same proof shows that
\[ |\re|(\Omega_\best(S,{\bf z}) \cap U) \cup \{\mu_u\} \]
is path connected.  In particular, we see that
\[W_0 = \left( W-\{[\mu_s]\} \right) \cup \{[\mu_u]\}\]
is a dense path connected subset disjoint from $[\mu_s]$, as required.
\end{proof}

\begin{remark}
By connecting an appropriately chosen pair of points to both the stable and the unstable foliations of the
pseudo-Anosov mapping class $\psi$, we can construct maps of circles into $\AF(S,{\bf z})$.  From these, one can find
embedded circles.
\end{remark}

\section{Closed surfaces}

The proof of Theorem \ref{mainarational} in the case ${\bf z} = \emptyset$ uses the case ${\bf z} \neq \emptyset$.  The
argument involves branched covers, so we begin with some elementary observations.

\subsection{Branched covers}

Let ${\bf z'} \subset S'$ be a finite set.  We say that a branched cover $f:S \to S'$ is \textbf{properly branched over
${\bf z'}$} if $f$ is a covering map from the complement of $f^{-1}({\bf z'})$ in $S$ to the complement of ${\bf z'}$
in $S'$ and if the restriction to any point of $f^{-1}({\bf z'})$ has no neighborhood on which $f$ is injective.  Said
differently, ${\bf z'}$ is precisely the branching locus in $S'$, and $f$ nontrivially branches at every point of
$f^{-1}({\bf z'})$.

\begin{proposition} \label{takingcover}
Suppose $f:S \to S'$ is a branched cover, properly branched over ${\bf z'} \subset S'$.  Then there is an embedding
$$f^*:\MF(S',{\bf z'}) \to \MF(S)$$
Moreover, $f^*(\AF(S',{\bf z'})) \subset \AF(S)$.
\end{proposition}
\begin{proof}
Given a measured foliation $(\F,\mu)$ on $(S',{\bf z'})$ there is a natural way of defining a measured foliation on
$S$, denoted $f^*(\F,\mu)$ as follows.  We define the underlying foliation $f^*(\F)$ so that the leaves are precisely
the preimages of leaves on $S'$.  The transverse measure, denoted $f^*(\mu)$ is defined by declaring the measure of an
arc on $S$ to be the measure of its image in $S'$. The $f$--image of a leaf cycle for $f^*(\F)$ can be used to
construct one for $\F$, and so $f^*(\AF(S',{\bf z'})) \subset \AF(S)$.

We must therefore show that $f^*$ is an embedding. One proof of this appeals to the theory of train tracks. We give a
different proof using quadratic differentials.

Fix a complex structure on $S'$ and one on $S$ so that $f:S \to S'$ is holomorphic.  According to the work of
Hubbard--Masur \cite{hubbardmasur} and Gardiner \cite{gardiner}, the space $Q(S',{\bf z'})$ of integrable meromorphic
quadratic differentials on $S'$ with the only poles at ${\bf z'}$ is naturally homeomorphic to $\MF(S',{\bf z'})$ (see
also Marden--Strebel \cite{mardenstrebel}). The homeomorphism is given by sending a quadratic differential $q \in
Q(S',{\bf z'})$ to the measure class of its vertical foliation $[\F(q),\mu(q)]$.  Likewise, the space of holomorphic
quadratic differentials $Q(S)$  is naturally homeomorphic to $\MF(S)$.

The pullback
$$f^*:Q(S',{\bf z'}) \to Q(S)$$
is an embedding, and one checks that
$$f^*(\F(q)) = \F(f^*(q)) \quad \mbox{ and } \quad f^*(\mu(q)) = \mu(f^*(q)).$$
It follows that
$$f^*:\MF(S',{\bf z'}) \to \MF(S)$$
is an embedding, as required.
\end{proof}

\subsection{Graph of involutions}

Let $\sigma$ be an involution of $S$ with nonempty fixed point set.  We write
$$f_\sigma:S \to S_\sigma = S/\langle \sigma \rangle$$
for the quotient.  If $f$ is properly branched over ${\bf z}_\sigma \subset S_\sigma$ (${\bf z}_\sigma \neq
\emptyset$), and $S_\sigma$ has genus at least $2$, then we say that $\sigma$ is an \textbf{allowable involution}.

Fix an allowable involution $\sigma$ and we define the \textbf{graph of $\sigma$--involutions} $\frak G_\sigma(S)$ as
follows. The vertex set of $\frak G_\sigma(S)$ is in a one-to-one correspondence with the $\Mod(S)$ conjugates of
$\sigma$. If $\sigma_0$ and $\sigma_1$ are conjugates of $\sigma$, then we connect the associated vertices (also
denoted $\sigma_0$ and $\sigma_1$) by an edge if and only if
\begin{itemize}
\item $G = \langle \sigma_0,\sigma_1 \rangle$ is a finite group and
\item the quotient $f_G:S \to S_G = S/G$ is properly branched over ${\bf z}_G$ and
$(S_G,{\bf z}_G)$ admits a pseudo-Anosov mapping class.
\end{itemize}

\begin{theorem} \label{involutiongraph}
For each closed surface $S$ of genus at least $4$, there exists an allowable involution $\sigma$ so that $\frak
G_\sigma(S)$ is connected.
\end{theorem}

We postpone the proof of this fact and use it to prove Theorem \ref{mainarational} for ${\bf z} = \emptyset$.
\begin{proof}[Proof of Theorem \ref{mainarational} for ${\bf z} = \emptyset$]
We are assuming that $S$ has genus at least $4$, and so according to Theorem \ref{involutiongraph} there exists an
allowable involution $\sigma$ so that $\frak G_\sigma(S)$ is connected.

Let $f_\sigma:S \to S_\sigma = S/\langle \sigma \rangle$ be the corresponding quotient properly branched over ${\bf
z}_\sigma$.  According to Proposition \ref{takingcover}, we have an embedding
$$f_\sigma^* : \AF(S_\sigma,{\bf z}_\sigma) \to \AF(S).$$
It follows from the case ${\bf z} \neq \emptyset$ of Theorem \ref{mainarational} that the subspace
$$X_\sigma = f_\sigma^*(\AF(S_\sigma,{\bf z}_\sigma))$$
is connected.  Since the involution $\sigma_0$ associated to any vertex of $\frak G_\sigma(S)$ is a conjugate of
$\sigma$, we also see that the associated space $X_{\sigma_0}$ is connected.

Observe that $\frak G_\sigma(S)$ admits an obvious action by $\Mod(S)$. Therefore, the set
$$\mathcal X(\sigma) = \bigcup_{\sigma_0 \in \ver(\frak G_\sigma(S))} X_{\sigma_0} \subset \AF(S)$$
is $\Mod(S)$--invariant and hence dense.

Now suppose $\{\sigma_0,\sigma_1\}$ is an edge of $\frak G_\sigma(S)$.  The hypothesis implies that $G = \langle
\sigma_0,\sigma_1 \rangle$ is a finite group and $S_G = S/G$ admits a pseudo-Anosov mapping class. A power of this can
be lifted to a pseudo-Anosov mapping class $\psi$ in $\Mod(S)$.  Moreover, because the branched covering $f_G$ factors
through the branched coverings
$$S \longrightarrow S_{\sigma_0} \longrightarrow S_G \quad \mbox{ and } \quad S \longrightarrow S_{\sigma_1} \longrightarrow S_G$$
we can assume that $\psi$
is also a lift of pseudo-Anosov mapping classes $\psi_0 \in \Mod(S_{\sigma_0},{\bf z}_{\sigma_0})$ and $\psi_1 \in
\Mod(S_{\sigma_1},{\bf z}_{\sigma_1})$.  As such, the stable fixed point of $\psi$ lies in $X_{\sigma_0} \cap
X_{\sigma_1}$. In particular, since these spaces are both connected and nontrivially intersect, it follows that
$X_{\sigma_0} \cup X_{\sigma_1}$ is connected.

Inductively we see that any two vertices $\sigma_0$ and $\sigma_1$ which are connected by an edge path have their
associated sets $X_{\sigma_0}$ and $X_{\sigma_1}$ in the same connected component of $\mathcal X(\sigma)$. Connectivity
of $\frak G_\sigma(S)$ means that every two vertices $\sigma_0$ and $\sigma_1$ are connected by an edge path, and so
$\mathcal X(\sigma)$ is connected. Therefore, $\overline {\mathcal X(\sigma)} = \AF(S)$ is connected.
\end{proof}

\subsection{Proof of Theorem \ref{involutiongraph}}

The proof of Theorem \ref{involutiongraph} divides into two cases depending on whether the genus is even or odd.  Both
proofs are essentially the same, except for the descriptions of the involutions.  We will first describe the involution
and the proof for the case of even genus (with corresponding figures for the case of genus $6$) and explain the proof
in detail. For odd genus, we simply describe the involution, with the remainder of the proof left as an easy exercise.\\

So, suppose that the genus $g$ of $S$ is even.   The dihedral group $D_g$ of order $g$ acts on $S$ with quotient
$f_{D_g}:S \to S_{D_g}= S/D_g$ having genus $1$ properly branched over ${\bf z}_{D_g}$ with $|{\bf z}_{D_g}| = 3$. See
Figure \ref{genus6dihedral} for the case of genus $6$.

\begin{figure}[htb]
\centerline{} \centerline{}
\begin{center}
\ \psfig{file=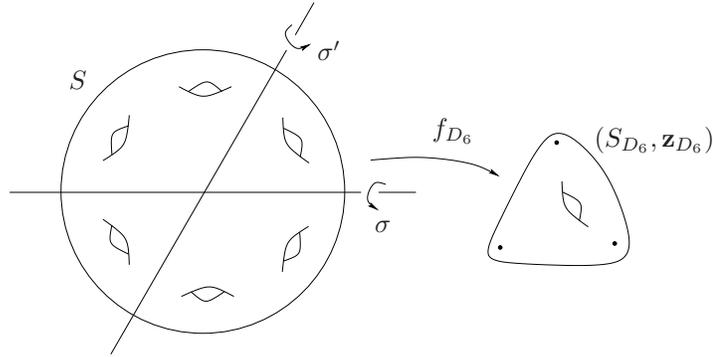,height=1.85truein} \caption{Generators $\sigma$ and $\sigma'$ for $D_6$ and the quotient of $S$.}
\label{genus6dihedral}
\end{center}
  \setlength{\unitlength}{1in}
  \begin{picture}(0,0)(0,0)
    \put(1.1,2){$S$}
    \put(2.7,1.25){$\sigma$}
    \put(2.4,2.15){$\sigma'$}
    \put(3.85,1.7){$(S_{D_6},{\bf z}_{D_6})$}
    \put(3,1.75){$f_{D_6}$}
  \end{picture}
\end{figure}

The group is generated by involutions $\sigma$ and $\sigma'$ which are conjugate in $D_g$ and hence also in $\Mod(S)$.
The quotient $f_\sigma:S \to S_\sigma = S/\langle \sigma \rangle$ has genus $g/2$ and is properly branched over ${\bf
z}_\sigma$ with $|{\bf z}_\sigma| = 2$.  See Figure \ref{genus6inv} for the case that $g = 6$.

\begin{figure}[htb]
\centerline{} \centerline{}
\begin{center}
\ \psfig{file=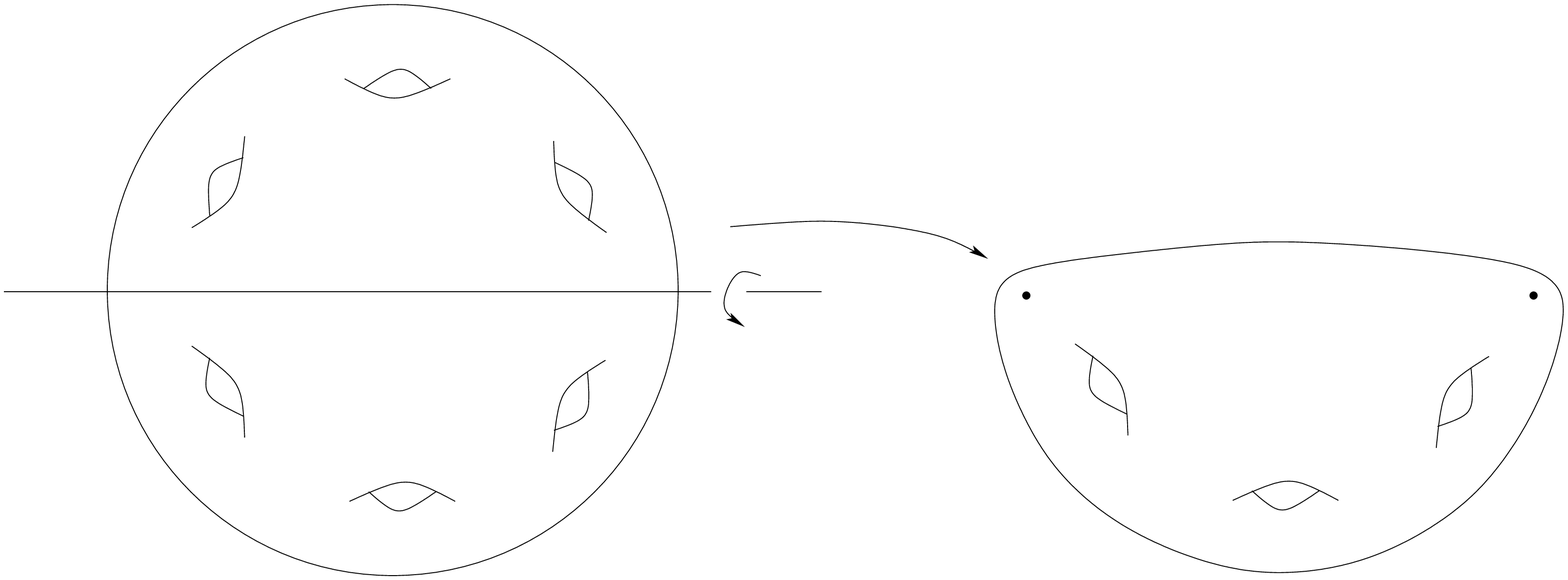,height=1.5truein} \caption{The involution $\sigma$ and the quotient of $S$.}
\label{genus6inv}
\end{center}
  \setlength{\unitlength}{1in}
  \begin{picture}(0,0)(0,0)
    \put(.6,1.8){$S$}
    \put(2.3,1.15){$\sigma$}
    \put(3.85,1.6){$(S_\sigma,{\bf z}_\sigma)$}
    \put(2.6,1.65){$f_\sigma$}
  \end{picture}
\end{figure}

We consider the involution graph $\frak G_\sigma(S)$.  According to the previous paragraph, $\sigma$ and $\sigma'$ are
both vertices of $\frak G_\sigma(S)$, and $\{\sigma,\sigma'\}$ is an edge.

\begin{theorem} \label{T:evencase}
For this choice of $\sigma$, $\frak G_\sigma(S)$ is connected.
\end{theorem}
\begin{proof}
Recall the Humphries generating set for the mapping class group---see \cite{humphriesgens}---consists of $2g+1$ Dehn
twists $T_{\gamma_1},....,T_{\gamma_{2g+1}}$.  The curves $\gamma_1,...,\gamma_{2g+1}$ can be chosen as in Figure
\ref{genus6gens}.  The relevant features are the following
\begin{enumerate}
\item $\gamma_g$ is invariant by $\sigma$,
\item $\gamma_{g-2}$ is invariant by $\sigma'$,
\item for every $i \neq g$, there exists an element of $\phi_i \in \Mod(S)$ commuting with $\sigma$ taking $\gamma_{g-2}$ to $\gamma_i$,
\end{enumerate}
That we may arrange for the last property is perhaps least obvious, but is an easy exercise given Figure
\ref{genus6gens}.

\begin{figure}[htb]
\centerline{} \centerline{}
\begin{center}
\ \psfig{file=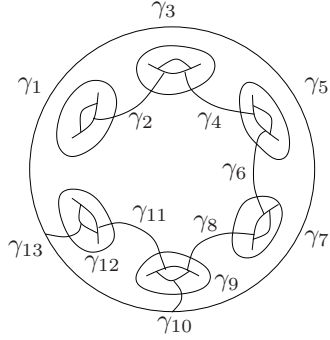,height=1.5truein} \caption{Twisting curves for the Humphries generators}
\label{genus6gens}
\end{center}
  \setlength{\unitlength}{1in}
  \begin{picture}(0,0)(0,0)
    \put(1.6,1.75){$\gamma_1$}
    \put(2.18,1.57){$\gamma_2$}
    \put(2.3,2.15){$\gamma_3$}
    \put(2.54,1.57){$\gamma_4$}
    \put(3.1,1.75){$\gamma_5$}
    \put(2.67,1.3){$\gamma_6$}
    \put(3.1,.95){$\gamma_7$}
    \put(2.52,1.03){$\gamma_8$}
    \put(2.63,.72){$\gamma_9$}
    \put(2.33,.5){$\gamma_{10}$}
    \put(2.2,1.08){$\gamma_{11}$}
    \put(1.95,.82){$\gamma_{12}$}
    \put(1.55,.9){$\gamma_{13}$}
  \end{picture}
\end{figure}

\begin{claim}
There is an edge path in $\frak G_\sigma(S)$ connecting the vertex $\sigma$ to the vertex $T_{\gamma_i} \sigma
T_{\gamma_i}^{-1}$ for any $i = 1,...,2g+1$.
\end{claim}
\begin{proof}[Proof of claim]
According to property (1), $T_{\gamma_g}$ commutes with $\sigma$.  It follows that $T_{\gamma_g} \sigma
T_{\gamma_g}^{-1} = \sigma$, so there is a constant path from $\sigma$ to $T_{\gamma_g} \sigma T_{\gamma_g}^{-1}$.

According to (3) for every $i \neq g$ there exists an element $\phi_i \in \Mod(S)$ commuting with $\sigma$ taking
$\gamma_{g-2}$ to $\gamma_i$.  Then $\phi_i T_{\gamma_{g-2}} \phi_i^{-1} = T_{\gamma_i}$.  Since $T_{\gamma_{g-2}}$
commutes with $\sigma'$ by (2), it follows that $T_{\gamma_i}$ commutes with $\phi_i \sigma' \phi_i^{-1}$. From this,
we see the following edges:
$$\begin{array}{rcl}
\{\phi_i \sigma \phi_i^{-1},\phi_i \sigma' \phi_i^{-1} \} & = & \{ \sigma, \phi_i \sigma' \phi_i^{-1} \}\\
\{T_{\gamma_i} \sigma T_{\gamma_i}^{-1},T_{\gamma_i} (\phi_i \sigma' \phi_i^{-1}) T_{\gamma_i}^{-1} \} & = &
\{T_{\gamma_i} \sigma T_{\gamma_i}^{-1},\phi_i \sigma' \phi_i^{-1} \}\\
\end{array}$$
It follows that $\sigma$ and $T_{\gamma_i} \sigma T_{\gamma_i}^{-1}$ are connected by an edge path (of length 2),
proving the claim.
\end{proof}

Let $\G(\Mod(S))$ denote the Cayley graph of $\Mod(S)$ with respect to the generating set
$\{T_{\gamma_1},...,T_{\gamma_{2g+1}}\}$.  It follows from the claim that there is a continuous equivariant map from
$\G(\Mod(S))$ to $\frak G_\sigma(S)$.  Since $\Mod(S)$ acts transitively on the vertices (by construction) and since
$\G(\Mod(S))$ is connected, it follows that $\frak G_\sigma(S)$ is connected, as required.
\end{proof}

When the genus $g$ is odd, we again choose an involution $\sigma$ for which it and a conjugate $\sigma'$ generate a
dihedral group, this time of order $2g$.  The involutions $\sigma$ and $\sigma'$ are shown in Figure
\ref{genus5dihedral} for the case of genus $5$.

\begin{figure}[htb]
\centerline{} \centerline{}
\begin{center}
\ \psfig{file=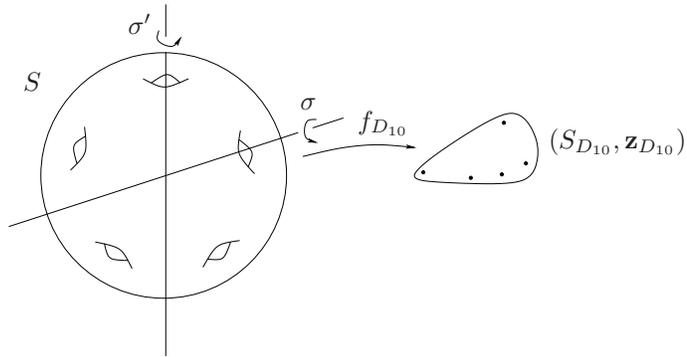,height=1.85truein} \caption{Generators $\sigma$ and $\sigma'$ for $D_{10}$ and the
quotient of $S$.} \label{genus5dihedral}
\end{center}
  \setlength{\unitlength}{1in}
  \begin{picture}(0,0)(0,0)
    \put(1.1,2){$S$}
    \put(1.65,2.3){$\sigma'$}
    \put(2.55,1.9){$\sigma$}
    \put(3.85,1.7){$(S_{D_{10}},{\bf z}_{D_{10}})$}
    \put(2.85,1.8){$f_{D_{10}}$}
  \end{picture}
\end{figure}

\begin{theorem} \label{T:oddcase}
With this choice of $\sigma$, $\frak G_\sigma(S)$ is connected.
\end{theorem}

The proof, which we omit, follows as in the previous case.  One need only verify that the Humphries generators for
$\Mod(S)$ can be chosen with similar properties to those used in the proof of Theorem \ref{T:evencase}.

\subsection{Cut points: ${\bf z} = \emptyset$}

To complete the proof of Theorem \ref{endinglams} in the case of ${\bf z} = \emptyset$, we must prove the counterpart
to Theorem \ref{nocut1}.

\begin{theorem} \label{nocut2}
The space $\AF(S)/\hspace{-.1cm} \sim$ has no cut points.
\end{theorem}
\begin{proof} We continue to denote the equivalence class of $\mu \in \AF(S)$ as $[\mu]$.
As in the proof of Theorem \ref{nocut1}, we will show that for every point $[\mu] \in \AF(S)/\hspace{-.1cm} \sim$,
there is a dense connected subset $W_0$ of $(\AF(S)/\hspace{-.1cm} \sim) -  \{[\mu]\}$.

We first extract the following key ingredients from the proof of the previous subsection.  The space $\AF(S)$ contains
a countable collection of connected subsets $\{\tilde W_j\}_{j =1}^\infty$ with the property that for any $i,j$ there
is a chain $\tilde W_j = \tilde W_{j_1},...,\tilde W_{j_n} = \tilde W_i$ so that $\tilde W_{j_k} \cap \tilde
W_{j_{k+1}}$ contains at least two points (the stable and unstable fixed points of some pseudo-Anosov mapping class)
and so that
\[ \tilde W = \bigcup_{j=1}^\infty \tilde W_j\]
is dense in $\AF(S)$.

Each of the sets $\tilde W_j$ is naturally homeomorphic to one of the connected spaces $\AF(S',{\bf z}')$.  In
particular, we see that the image of $\tilde W_j$ in $\AF(S)/\hspace{-.1cm} \sim$ is a connected set $W_j$ with no cut
points. Moreover, the collection $\{W_j\}_{j=1}^\infty$ has the same property for $\AF(S)/\hspace{-.1cm} \sim$ as
$\{\tilde W_j\}_{j=1}^\infty$ has for $\AF(S)$ as described in the previous paragraph.  We set
\[ W = \bigcup_{j=1}^\infty W_j. \]

To find the required set $W_0$, we first note that if $[\mu] \not \in W$, then we can take $W_0 = W$.  We therefore
assume that $[\mu] \in W$.

Since $W_j$ is connected and has no cut points for each $j \geq 1$, we see that $W_j - \{ [\mu]\}$ is connected for
every $j \geq 1$.  Moreover, if $W_i \cap W_j$ contains at least two points, then
\[(W_i - \{ [\mu]\})  \cap ( W_i - \{ [\mu] \}) \]
contains at least $1$ point.  It follows that
\[W_0 = \bigcup_{j=1}^\infty (W_j - \{ [\mu]\}) = \left( \bigcup_{j=1}^\infty W_j \right) - \{ [\mu]\} = W - \{[\mu]\}\]
is connected.  Since $W$ is dense in $\AF(S)/\hspace{-.1cm} \sim$, it follows that $W_0$ is dense in
$(\AF(S)/\hspace{-.1cm} \sim) - \{[\mu]\}$, as required.
\end{proof}

\section{Hyperbolic $3$--manifolds} \label{hyperbolicgeometry}

For the purposes of this section, we take $\Sigma = \Sigma_{g,n}$ to be a compact orientable surface of genus $g$ with
$n$ boundary components.  Here we will work exclusively with laminations instead of foliations as these are more natural in this context.   We will write $\lambda$ to denote a measured lamination (or its projective class) and $|\lambda|$ for the supporting (geodesic) lamination.

Given a hyperbolic $3$--manifold $M = \mathbb H^3/\Gamma$ and $\epsilon > 0$, we let $M_\epsilon$ be the
$\epsilon$--thick part of $M$.  We fix $\epsilon > 0$ less than the $3$--dimensional Margulis constant so that the
$\epsilon$--thin part, $M- M_\epsilon$, is a disjoint union of Margulis tubes and parabolic cusps.  The Margulis tubes
are homeomorphic to open solid tori and each of the parabolic cusps is homeomorphic to either $A \times (0,\infty)$ or
$T \times (0,\infty)$, where $A$ and $T$ denote an open annulus and torus, respectively (for example, see
\cite{thurstonbook} or \cite{benepetro}). We further define $M^0 = M_\epsilon^0 = M_\epsilon \cup \, \{\mbox{Margulis
tubes} \}$.  The boundary of $M^0$ is thus a union of annuli and tori and $M^0$ is homotopy equivalent to $M$.

For $\Gamma = \pi_1(M)$ finitely generated, it follows from \cite{scottcore} that there exists a compact core $N
\subset M^0$ for which the inclusion is a homotopy equivalence.  Indeed, from
\cite{mcculloughrelcore},\cite{kulshalrelcore} one may choose such an $N$ so that
\[P = N \cap \partial M^0\]
is a union of incompressible annuli and tori called the parabolic locus of $N$.  Moreover, $N$ can be chosen so that
each component $U \subset \overline{M^0-N}$ is a neighborhood of the unique end of $M^0$ it contains and so that the
inclusion $U \cap N \hookrightarrow U$ induces a homotopy equivalence.   Observe that $U \cap N$ is a component of $\partial N-P$, and we say that this component faces $U$.
We call such an $N$ a relative compact core.

The space of equivalence classes of orientable hyperbolic $3$--manifolds marked by a relative homotopy equivalence to
$(\Sigma,\partial \Sigma)$ is denoted
\[\AH(\Sigma,\partial \Sigma) = \{ f:(\Sigma,\partial \Sigma) \longrightarrow^{\hspace{-.4cm}\simeq} \hspace{.25cm} (M_0,\partial M_0) \, | \, M = \mathbb H^3/\Gamma \}/\hspace{-.1cm} \sim \, \, .\]
Here $(f:(\Sigma,\partial \Sigma) \to (M^0,\partial M^0)) \sim (h:(\Sigma,\partial \Sigma) \to (L^0,\partial L^0))$ if
there exists an isometry $\varphi:M \to L$ and a relative homotopy $\varphi \circ f \simeq h$. Using the holonomy
homomorphism we can identify a point of $\AH(\Sigma,\partial \Sigma)$ with a conjugacy class of homomorphisms to
$\PSL_2(\mathbb C) \cong \Isom^+(\mathbb H^3)$.  We thus view $\AH(\Sigma,\partial \Sigma)$ as a subspace of the space
of conjugacy classes of homomorphisms to $\PSL_2(\mathbb C)$ which defines a topology on $\AH(\Sigma,\partial \Sigma)$
(the algebraic topology). We abuse notation and simply denote points of $\AH(\Sigma,\partial \Sigma)$ as $M$, with the
marking understood unless clarification is necessary.

For $M \in \AH(\Sigma,\partial \Sigma)$, work of Bonahon \cite{bonahonends} and Thurston \cite{thurston} implies that
$M^0 \cong \Sigma \times \mathbb R$, and $N$ can be chosen to correspond to $\Sigma \times [-1,1]$ under this
homeomorphism.  We choose $\epsilon$ sufficiently small so that given any simple closed curve $\gamma \subset \Sigma$,
the geodesic representative of $\gamma$ lies in $M^0$ (which is possible since the geodesic representative lies on a
pleated surface and so cannot penetrate too far into any cusp).

Given $M \in \AH(\Sigma,\partial \Sigma)$ and a simple closed curve $\gamma \in \C^0(\Sigma)$ one can measure the
length of $\gamma$ in $M$, which is the length of the geodesic representative, unless none exists (in which case the
length is zero). In \cite{brocklength}, Brock proved that this naturally extends to a continuous function.
\begin{theorem}[Brock] \label{continuouslength}
There is a continuous function
\[ \Len : \AH(\Sigma,\partial \Sigma) \times \ML(\Sigma) \to \mathbb R \]
homogeneous in the second argument which extends the above mentioned length function on $\AH(\Sigma,\partial \Sigma) \times \C^0(\Sigma)$.
\end{theorem}
In what follows, we assume $M \in \AH(\Sigma,\partial \Sigma)$ has no accidental parabolics, meaning that the only
parabolics in $\pi_1(M)$ are represented by peripheral loops in $\Sigma$.  Given $M \in \AH(\Sigma,\partial \Sigma)$,
we let $C(M)$ denote the convex core of $M$. Fix a component $U \subset \overline{M^0 - N}$ providing a neighborhood of
an end of $M^0$. This end is called geometrically finite if $U \cap C(M)$ is compact, and simply degenerate otherwise.

The work of Thurston and Bonahon greatly clarifies the simply degenerate ends as we now describe. A sequence of simple
closed curves $\{\gamma_i\}_{i=1}^\infty$ is said to exit the end if the geodesic representative $\gamma_i^*$ of
$\gamma_i$ in $M$ is contained in $U$ for every $i$ and for every compact $K \subset M$ there are at most finitely many
$i$ for which $\gamma_i^* \cap K \neq \emptyset$ (that is, the geodesics $\gamma_i^*$ lie further and further out in
$U$).

The work of Thurston and Bonahon associates a unique lamination $|\lambda| \in \EL(\Sigma)$ with the following
property. First, there exists a sequence of simple closed curves $\{\gamma_i\}_{i=1}^\infty$ so that in $\PML(\Sigma)$
we have $\displaystyle{\lim_{i \to \infty} \gamma_i = \lambda'}$ for some $\lambda' \in \PML(\Sigma)$ with $|\lambda'| = |\lambda|$.
Second, for every sequence $\{ \gamma_i\}_{i=1}^\infty$ exiting the end, up to subsequence, $\displaystyle{\lim_{i \to
\infty}\gamma_i = \lambda'}$ with $|\lambda'| = |\lambda|$.  This lamination $|\lambda|$ is called the ending lamination
associated to the simply degenerate end.

Using the orientations on $M$ and on $\Sigma$ the ends of $M^0$ can be labeled as either positive or negative.  We
define $\E^+(M)$ and $\E^-(M)$ to be the ending laminations associated to the positive and negative ends of $M^0$,
respectively, if the end is simply degenerate.  If either or both of the ends are geometrically finite, then we let
$\E^\pm(M) \in \T(\Sigma)$ be the (finite type) conformal structure at infinity associated to the geometrically finite
end.

We say that $M$ is doubly degenerate if both ends are simply degenerate.  We denote the space of doubly degenerate
manifolds $\DD(\Sigma,\partial \Sigma) \subset \AH(\Sigma,\partial \Sigma)$.  Fixing a conformal structure $Y \in
\T(\Sigma)$, we denote the associated Bers slice
\[ B_Y = \{ M \in \AH(\Sigma,\partial \Sigma) \, | \, M \mbox{ is geometrically finite, and } \E^-(M) = Y \}. \]
The singly degenerate points on the boundary of the Bers slice are denoted
\[ \partial_0 B_Y = \{ M \in \overline{B_Y} \, | \, \E^-(M) = Y \, , \, \, \E^+(M) \in \EL(\Sigma) \}. \]

We thus have maps
\[\E:\DD(\Sigma,\partial \Sigma) \to \EL(\Sigma) \times \EL(\Sigma) - \Delta \]
and
\[\E^+:\partial_0 B_Y \to \EL(\Sigma).\]

A result of Thurston \cite{thurston,thurstonfiber} states that the length function can be used to characterize the
ending laminations (as a set).  This is described by the following corollary of Theorem \ref{continuouslength}.
\begin{corollary}[Thurston] \label{continuouslengthcor}
For $M \in \DD(\Sigma,\partial \Sigma)$, $\Len_M(\lambda) = 0$ if and only if
\[ |\lambda| \in \{\E^+(M),\E^-(M) \}. \]
Similarly, for $M \in \partial_0 B_Y$, $\Len_M(\lambda) = 0$ if and only if $|\lambda| = \E^+(M)$.
\end{corollary}

Surjectivity of $\E$ onto $\EL(\Sigma) \times \EL(\Sigma) - \Delta$ (and $\E^+$ onto $\EL(\Sigma)$) follows from the
Bers Simultaneous Uniformization Theorem \cite{berssimunif}, Thurston's Double Limit Theorem \cite{thurstonfiber} (see
also Otal \cite{otaldlt}), and Corollary \ref{continuouslengthcor} (to guarantee the correct ending laminations; see also the proof of Theorem \ref{ELThomeo1}).  A
slightly stronger version of the Double Limit Theorem which we will need is given by the following.

\begin{theorem}[Thurston] \label{DLT} Suppose $\{M_i\} \in \AH(\Sigma, \partial \Sigma)$ has the property that for some $K > 0$
\[ \Len_{M_i}(\lambda_i) + \Len_{M_i}(\mu_i) \leq K \]
for some $\{\lambda_i\},\{\mu_i\} \subset \ML(\Sigma)$.  If $\lambda_i \to \lambda$, $\mu_i \to \mu$ and $\lambda,\mu$
fills $S$, then up to subsequence, $M_i$ converges to some manifold $M \in \AH(\Sigma,\partial \Sigma)$.
\end{theorem}
This is precisely Theorem 6.3 of \cite{thurstonfiber}, but also follows from a diagonal argument, the Density Theorem
\cite{ELCII} and the Double Limit Theorem as stated in Theorem 5.0.1 of \cite{otaldlt}.

Our discussion culminates in the following simplified version of the Ending Lamination Theorem \cite{ELCI},\cite{ELCII}
for $\DD(\Sigma,\partial \Sigma)$ and $\partial_0 B_Y$.
\begin{theorem}[Brock--Canary--Minsky] \label{ELT} $\E$ is a bijection on $\DD(\Sigma,\partial \Sigma)$ and $\E^+$ is
a bijection on $\partial_0 B_Y$.
\end{theorem}

We can now assemble the pieces from the discussion above to prove that the ending laminations actually serve as a
continuous parameterization of $\DD(\Sigma,\partial \Sigma)$.

\begin{theorem} \label{ELThomeo1}  The map $\E:\DD(\Sigma,\partial \Sigma) \to \EL(\Sigma) \times \EL(\Sigma)- \Delta$ is a homeomorphism.
\end{theorem}
\begin{proof} As just noted, Theorem \ref{ELT} implies that $\E$ is a bijection.  We are therefore left to prove that $\E$ and
$\E^{-1}$ are continuous.

We first show that $\E$ is continuous.  Suppose $\{M_i\} \subset \DD(\Sigma,\partial \Sigma)$ is any sequence which
converges to some $M \in \DD(\Sigma,\partial \Sigma)$.  It suffices to show that there is a subsequence, also called
$\{M_i\}$ so that $\E^\pm(M_i) \to \E^\pm(M)$ as $i \to \infty$.  Throughout, when passing to subsequences, we always
re-index so that the index set is the set of positive integers.

We start by picking a subsequence of measured laminations $\{\lambda_i^\pm\}$ with $|\lambda_i^\pm| = \E^\pm(M_i)$
which converge to some measured laminations $\lambda^\pm$.  We wish to show that $\lambda^\pm$ are measures supported
on $\E^\pm(M)$.

Theorem \ref{continuouslength} and Corollary \ref{continuouslengthcor} implies
\[\lim_{i \to \infty} \Len_{(M_i)}(\lambda^\pm_i) = \Len_M(\lambda^\pm) = 0. \]
Corollary \ref{continuouslengthcor} then implies that one of the following happens
\[\lim_{i \to \infty} \E^\pm(M_i) \to \E^\pm(M) \, \, \mbox{ or } \, \, \lim_{i \to \infty} \E^\pm(M_i) \to \E^\mp(M).\]
\begin{claim} $\displaystyle{\lim_{i \to \infty} \E^\pm(M_i) \to \E^\pm(M)}$.
\end{claim}
\begin{proof}
The argument we give here is a very minor modification of an argument
given by Brock--Bromberg in the proof of Theorem 8.1 of \cite{brobrodense}.

We suppose that
\[\lim_{i \to \infty} \E^\pm(M_i) \to \E^\mp(M) \]
and arrive at a contradiction.

Since there are no accidental parabolics, work of Canary \cite{canary} implies that the convergence $M_i \to M$ is strong, and so in particular there are
relative compact cores $N_i \subset M_i^0$ and $N \subset M^0$ and $K_i$--bi-Lipschitz maps
\[F_i: N_i \to N \]
compatible with the markings with $K_i \to 1$ as $i \to \infty$.  Let $\overline{M_i^0 - N_i} = U_i^- \sqcup U_i^+$
denote the corresponding neighborhoods of the negative and positive ends of $M_i^0$, respectively. Similarly, we write
$\overline{M^0 - N} = U^- \sqcup U^+$.  We remark that $F_i$ takes the component of $N_i - P_i$ facing $U_i^\pm$ to the
component of $N-P$ facing $U^\pm$.

Let $\{\gamma_i\} \subset \C(\Sigma)$ be a sequence exiting the end of $M^0$ defined by $U^-$ so that
\[\lim_{i \to \infty} \gamma_i = \E^-(M)\]
in the quotient of $\PML(\Sigma)$ obtained by forgetting measures.  It follows from Klarreich's work \cite{klarreich}
that this convergence also takes place in $\overline \C(\Sigma)$.

Strong convergence further implies that, after passing to a subsequence if necessary, we may assume that the geodesic
representative of $\gamma_i$ in $M_i$ lies in $U_i^-$.

Similarly, for each $i$, we can construct a sequence of curves $\{\delta_j(i)\}_{j=1}^\infty$ exiting the end of
$M_i^0$ defined by $U_i^+$ and hence converging to $\E^+(M_i)$.  Since $\E^+(M_i) \to \E^-(M)$, we can choose a
diagonal sequence $\{ \delta_i \}_{i=1}^\infty$ with $\delta_i = \delta_{j(i)}(i)$ so that
\[ \lim_{i \to \infty} \delta_i = \E^-(M)\]
in $\overline \C(\Sigma)$.

Next observe that if $\{ \alpha_i\}_{i=1}^\infty \subset \C(\Sigma)$ is any sequence with $\alpha_i$ a vertex of the
geodesic $[\gamma_i,\delta_i] \subset \C(\Sigma)$, then we also have
\[\lim_{i \to \infty} \alpha_i = \E^-(M).\]

Now, fix any $i$.  Two consecutive vertices of $[\gamma_i,\delta_i]$ represents a pair of disjoint essential simple
closed curves, and so can be realized in a single pleated surface in $M_i$ (see \cite{thurston} and \cite{ceg}).  The
geodesic $[\gamma_i,\delta_i]$ thus gives rise to a finite set of pleated surfaces $X_i(1),...,X_i(d_i)$ where $d_i =
d(\gamma_i,\delta_i)$ with the property that
\begin{enumerate}
\item $\gamma_i$ is realized on $X_i(1)$,
\item $\delta_i$ is realized on $X_i(d_i)$, and
\item $X_i(j) \cap X_i(j+1) \neq \emptyset$ for each $j = 1,..,d_i-1$.
\end{enumerate}
Therefore, there is a vertex $\alpha_i \in [\gamma_i,\delta_i]$ realized on one of these pleated surfaces in $M_i$,
call it $X_i$, which nontrivially intersects $N_i$.

Once again appealing to strong convergence we see that pleated surfaces $Y_i$ in $M$ realizing $\alpha_i$ must
intersect the (compact) $R$--neighborhood of $N$ for some $R > 0$.  Compactness of pleated surfaces implies that $Y_i$
converges to a pleated surface $Y$ in $M$ which realizes the limit $\E^-(M)$ of $\alpha_i$.  This is a contradiction,
and it follows that $\E^\pm(M_i) \to \E^\pm(M)$, verifying the claim. \end{proof}

All that remains is to show that $\E^{-1}$ is continuous, and this will complete the proof of the theorem.  We assume
that $\{M_i\} \subset \DD(\Sigma,\partial \Sigma)$ with
\[\lim_{i \to \infty} \E^\pm(M_i) = \E(M)\]
and prove that, after passing to a subsequence if necessary, we have
\[\lim_{i \to \infty} M_i = M.\]
This will show that $\E^{-1}$ is continuous.

Let $\{\lambda_i^\pm\}_{i=1}^\infty$ be any measured laminations with $|\lambda_i^\pm| = \E^\pm(M_i)$.  After scaling
the transverse measures and passing to a subsequence if necessary, we may assume that
\[ \lim_{i \to \infty} \lambda_i^\pm = \lambda^\pm \]
where $|\lambda^\pm| = \E^\pm(M)$.

Theorem \ref{DLT} implies that, after passing to further subsequence if necessary, $\{M_i\}$ converges to some $M' \in
\AH(\Sigma,\partial \Sigma)$.  Since $\Len_{M_i}(\lambda_i^\pm) = 0$, it follows that $\Len_{M'}(\lambda^\pm) = 0$.
Therefore, $\{E^+(M),\E^-(M)\} = \{\E^+(M'),\E^-(M')\}$, so either $\E^\pm(M) = \E^\pm(M')$ and we are done by Theorem
\ref{ELT}, or else $\E^\pm(M) = \E^\mp(M')$.  The argument given above implies the latter situation cannot occur, and
so the proof is complete.
\end{proof}

In a similar fashion, one obtains
\begin{theorem} \label{ELThomeo2}
The map $\E^+:\partial_0 B_Y \to \EL(\Sigma)$ is a homeomorphism.
\end{theorem}

The proof is similar to the previous proof but simpler and we leave it to the reader. We note that strong convergence
in this setting follows from Anderson--Canary \cite{andersoncanary}.

\bibliographystyle{plain}
\bibliography{boundary}

\bigskip

\noindent Department of Mathematics, University of Illinois, Urbana--Champaign, IL 61801 \newline \noindent
\texttt{clein@math.uiuc.edu}

\bigskip

\noindent Department of Mathematics, University of Warwick, Coventry CV4 7AL UK
\newline \noindent \texttt{s.schleimer@warwick.ac.uk }

\end{document}